\documentclass[11pt]{amsart}
\usepackage{amsmath}
\usepackage{amsfonts,amssymb,amsthm,amstext,amscd,boxedminipage}  

\usepackage{graphicx}
\usepackage {psfrag} 
\usepackage{color}
 
 \numberwithin{equation}{section}

\DeclareMathSymbol{\minus} {\mathord}{operators}{"2D} %

\theoremstyle{plain}
\newtheorem{theorem}{Theorem}[section]
\newtheorem{lem}[theorem]{Lemma}

\newtheorem{pro}[theorem]{Proposition}

\newtheorem{alg}[theorem]{Algorithm}

\theoremstyle{definition}
\newtheorem{df}[theorem]{Definition}
\newtheorem{remark}[theorem]{Remark}
\newtheorem{example}[theorem]{Example}

\def \R {\mathbb{R}}
\def \C {\mathbb{C}}
\def \H {\mathbb{H}}
\def \Z {\mathbb{Z}}
\def \D {\mathbb{D}}
\def \M {\mathcal M}
\def \CN {\mathbb N}
\def \CO {\mathcal O}
\def \CR {\mathcal R}
\def \CS {\mathcal S}
\def \CW {\mathcal W}
\def \rk {{\mbox rank}}

\begin{document}

\title[Twisted dimer model]{A twisted dimer model for knots}
\date{\today}

\author[M. Cohen]{Moshe Cohen}
\address{Department of Mathematics, Bar Ilan University,
Ramat Gan 52900, Israel}
\email{cohenm10@macs.ac.il}
\thanks {}

\author[O. T. Dasbach]{Oliver T. Dasbach}
\address{Department of Mathematics, Louisiana State University,
Baton Rouge, LA 70803, USA}
\email{kasten@math.lsu.edu}
\thanks {The second author was supported in part by NSF grants DMS-0806539 and DMS-0456275 (FRG). The first and third authors were partially supported by NSF VIGRE grant DMS 0739382. }
 
\author[H. M. Russell]{Heather M. Russell}
\address{Department of Mathematics, Louisiana State University,
Baton Rouge, LA 70803, USA}
\email{hrussell@math.lsu.edu}
\thanks {}

\begin{abstract}
We develop a dimer model for the Alexander polynomial of a knot. This recovers Kauffman's state sum model for the Alexander polynomial using the language of dimers. By providing some additional structure we are able to extend this model to give a state sum formula for the twisted Alexander polynomial of a knot depending on a representation of the knot group. 
\end {abstract}
 
\maketitle

\def\kbsm#1{\mathscr{S}_K(#1)}
\def\sgn#1;#2{\mathbb{S}_{#1,#2}} 
\def\mcg#1;#2{\Gamma_{#1,#2}} 
\def\fg#1;#2{\Pi_{#1,#2}}
\def\tb#1;#2{\mathscr{K}_{\frac{#1}{#2}}}
\def\periph{(\mathcal{\mu},\mathcal{\lambda})}
\def\ext#1{\mathscr{E}(\mathscr{#1})}

%% Paper Notation
\def \R {\mathbb{R}}
\def \C {\mathbb{C}}
\def \H {\mathbb{H}}
\def \Z {\mathbb{Z}}
\def \D {\mathbb{D}}
\def \M {\mathcal M}
\def \CN {\mathbb N}
\def \CO {\mathcal O}
\def \CR {\mathcal R}
\def \CS {\mathcal S}
\def \CW {\mathcal W}
\def \rk {{\mbox rank}}
\def \BRT {{Bollob\'as--Riordan--Tutte }}
\def \Vol {\mathrm{Vol}}
\def \tr {\mathrm{tr}}

\def\frametitle {}
\def\block {}

\section{Introduction}
A dimer is an edge in a bipartite graph, and a dimer covering is a perfect matching for that graph. The study of dimer coverings started in the 1960's with the work of Kasteleyn \cite{Kasteleyn:DimerModel} and Temperley-Fisher \cite{TemperleyFisher:Dimers} who used it as a tool for studying  statistical physics. 
Kasteleyn showed that the partition function on weighted bipartite planar graphs can be expressed as a determinant of a suitable matrix.
The last ten years have seen a resurgence of the study of dimers and the application of this theory to many other areas of mathematics. 

Our interest is in exploring the opposite direction. We have given matrices, and we want to find the corresponding dimer model that expresses the determinant of the matrix as the partition function on the graph.  Two well-known polynomial knot invariants, the classical Alexander polynomial and Xiao-Song Lin's twisted Alexander polynomial, are defined as determinants.
The goal of this current work is to use the language of dimers to find a combinatorial model for the Alexander polynomial and the twisted Alexander polynomial. 

Given a knot $K$ in $S^3$ and some generic diagram for the knot we construct an associated planar bipartite graph with one set of vertices corresponding to crossings and the other set corresponding to faces. Edges signify incidence between crossings and faces.  Using this graph along with a certain weighting of the edges we provide a state sum model for the Alexander polynomial in terms of dimer coverings. This model recovers Kauffman's  state sum model for the Alexander polynomial \cite{Kauffman:OnKnots}. 

Consider a representation $\rho$ of the fundamental group of the knot complement. Associated to this representation one defines the twisted Alexander polynomial $\Delta_{\rho, K}$ which is an invariant of the knot together with the representation $\rho$.  We extend our dimer model for the Alexander polynomial to provide a state sum model for the twisted Alexander polynomial. 

In Section \ref{Dimer} we review some basic definitions and theorems dealing with dimer coverings. In Section 3 we recall the definition of the Alexander polynomial and show how to translate it into a dimer model. We will see that it is equivalent to Kauffman's state sum model. Section 4 begins with the definition of the twisted Alexander polynomial and shows how the dimer model can be augmented to provide a twisted dimer model. We provide examples throughout.

\section{Dimer Background} \label{Dimer}
In this section we review some facts and results about dimer coverings of graphs. For the interested reader Kenyon provides an excellent introductory set of lectures on this subject  \cite{Kenyon:LecturesOnDimers}. We will also need to recall some results of Kasteleyn \cite{Kasteleyn:DimerModel} a good explanation of which can be found in Kuperberg's work \cite{Kuperberg:Symmetries, Kuperberg:KasteleynCokernels, Kuperberg:PermanentDeterminant}. 

Let $\Gamma = (V_1, V_2, E)$ be a 
bipartite graph with $V_1$ and $V_2$ the two vertex sets and $E$ the collection of edges in $\Gamma$ each of which has one endpoint in $V_1$ and one endpoint in $V_2$. 

\begin{df}
A {\bf dimer} is an edge in $E$. A {\bf dimer covering} is a subset $m$ of $E$ such that each vertex in $\Gamma$ is an endpoint of exactly one edge in $m$. In other words a dimer covering is a perfect matching on $\Gamma$. 
Let $\mathcal{M}$ be the set of all dimer coverings of  $\Gamma$; note that $\mathcal{M} = \emptyset$ whenever $|V_1| \neq |V_2|$.
\end{df}

%\begin{figure}[h]
%\includegraphics[width=1.3in]{wackmatch.eps}
%\caption{A dimer covering of a bipartite graph.} \label{dimerex}
%\end{figure}

Let $\mu: E \rightarrow \mathbb{C}[t]$ be a weighting of the graph $\Gamma$, and denote the weighted graph by $\Gamma_{\mu}$.   Then we consider the following partition function $Z(\Gamma_{\mu})$ which is of particular interest in statistical physics.

$$Z(\Gamma_{\mu})=\sum_{m \, \in \, \mathcal{M}} \prod_{e \, \in \, m} \mu(e)$$ 

\begin{df} 
Given a weighting $\mu$ of $\Gamma$ and an ordering of the vertex sets 
$$V_1 = \{ v_{1,1}, \ldots, v_{|V_1|, 1} \} \mbox{ and } V_2 = \{ v_{1, 2}, \ldots, v_{|V_2|,2}\}$$ construct the matrix $M(\Gamma_{\mu})$ of dimension $|V_1| \times |V_2|$ with entries specified by the weight function $\mu$ as follows: The $ij^{th}$ entry of $M(\Gamma_{\mu})$ is given by the sum of all weights $\mu$ assigned to edges between $v_{i,1}$ and $v_{j,2}$. We call this the {\bf weight matrix} for $\Gamma_{\mu}$. Figure \ref{kastmat} gives an example.
\end{df}

Let $\textup{Perm}(M(\Gamma_{\mu}))$ denote the permanent (or unsigned determinant) of $M(\Gamma_{\mu})$. 
Thus, we see that $Z(\Gamma_{\mu}) = \textup{Perm}(M(\Gamma_{\mu}))$. In the case that $|V_1| \neq |V_2|$ both values are 0. A natural question to ask is: under what conditions can the weighting $\mu$ be modified to get a new weighting $\mu'$ with the property that the partition function for $\Gamma_{\mu'}$ is the determinant of the weight matrix for $\Gamma_{\mu}$. In other words does there exist weighting $\mu': E \rightarrow \mathbb{C}[t]$ such that
$Z(\Gamma_{\mu'}) = \textup{Det}(M(\Gamma_{\mu}))?$ In the case that $\Gamma$ is planar Kasteleyn proves that such a modification is always possible. He accomplishes this by using what is now called a Kasteleyn weighting.

\begin{df}
Let $\Gamma$ be a bipartite plane graph, that is a bipartite graph together with a fixed embedding of that graph in the plane. A {\bf Kasteleyn weighting} $\epsilon: E \rightarrow \{ \pm 1\}$ is a choice of $\pm 1$ for each edge with the property that each bounded face with 0 mod 4 edges has an odd number of $-1$ assignments and each bounded face with 2 mod 4 edges has an even number of $-1$ assignments.
\end{df} 

\begin{pro}[Kasteleyn] \label{Kweight}
Every  bipartite plane graph $\Gamma$ has a Kasteleyn weighting. 
\end{pro}

\begin{proof}
We can prove this fact by providing an algorithm for finding a Kasteleyn weighting. Begin by choosing a spanning tree $T = (V_1, V_2, E_T)\subset \Gamma$. If $\bar{\Gamma}$ is the dual plane graph of $\Gamma$ (with vertices given by faces of $\Gamma$, edges transverse to edges of $\Gamma$, and faces given by vertices of $\Gamma$) then there is an associated spanning tree $\bar{T}\subset \bar{\Gamma}$ which is disjoint from $T$. Consider $\bar{T}$ to be rooted at the vertex corresponding to the unbounded face.

Let $\epsilon: E_T \rightarrow \{\pm 1\}$ be arbitrarily given. Choose a valence 1 vertex $v_1$ of $\bar{T}$ that is not the root vertex. This vertex represents a face of $\Gamma$ that has all but one bounding edge present in the tree $T$. Let $e_1$ be the edge that is missing in that face of $\Gamma$. Since all other edges bounding the face have been assigned weights, the choice of value for $\epsilon(e_1)$ that will satisfy the properties of a Kasteleyn weighting is forced. Remove the vertex $v$ and the edge incident on $v$ from $\bar{T}$, and define $\epsilon(e_1)$ as necessary.

Repeat this process, pruning the non-root valence 1 vertices and their edges as you go. Eventually a single edge connecting a vertex $v$ to the root is all that remains in $\bar{T}$. This represents a single edge $e$ in $\Gamma$ shared by a bounded face and the unbounded face. Assign the necessary value to $\epsilon(e)$ in order to complete to a Kasteleyn weighting.
\end{proof}
 
\begin{df}
Let $\Gamma_{\mu}$ be a weighted bipartite plane graph. We call the weight matrix $M(\Gamma_{\epsilon \cdot \mu})$ the {\bf Kasteleyn matrix}, and we give it the special notation $K(\Gamma_{\mu}).$ See Figure \ref{kastmat} for a calculation of a Kasteleyn matrix.
\end{df}

\begin{figure}[ht]
\scalebox{.75}{\begin{picture}(60,90)(-70,-40)
\put(-110,20){$\Gamma_{\mu} = \hspace{10pt} \raisebox{-30pt}{\includegraphics[width=1.5in]{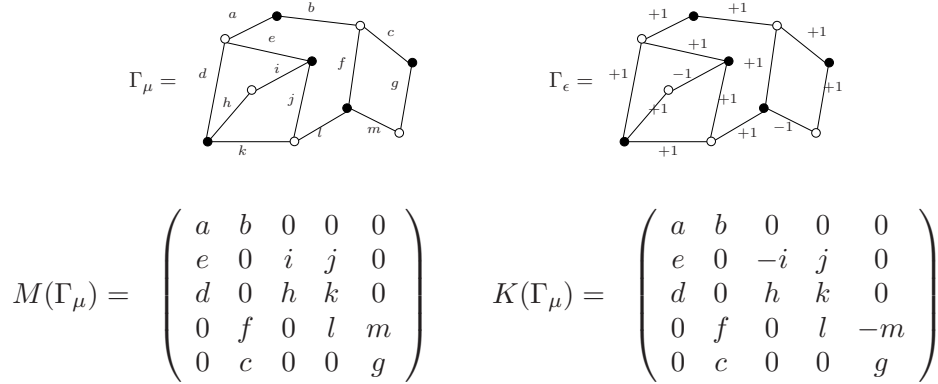}}$}
\put(-60,55){\tiny{$a$}}
\put(-20,58){\tiny{$b$}}
\put(20,46){\tiny{$c$}}
\put(-75,25){\tiny{$d$}}
\put(-40,43){\tiny{$e$}}
\put(-5,30){\tiny{$f$}}
\put(22,20){\tiny{$g$}}
\put(-63,10){\tiny{$h$}}
\put(-37,27){\tiny{$i$}}
\put(-30,12){\tiny{$j$}}
\put(-55,-14){\tiny{$k$}}
\put(-15,-5){\tiny{$l$}}
\put(10,-3){\tiny{$m$}}
\end{picture}
\hspace{2in}
\begin{picture}(60,90)(-70,-40)
\put(-110,20){$\Gamma_{\epsilon} = \hspace{10pt} \raisebox{-30pt}{\includegraphics[width=1.5in]{wackweight.eps}}$}
\put(-60,55){\tiny{$+1$}}
\put(-20,58){\tiny{$+1$}}
\put(20,46){\tiny{$+1$}}
\put(-80,25){\tiny{$+1$}}
\put(-40,40){\tiny{$+1$}}
\put(-12,30){\tiny{$+1$}}
\put(28,18){\tiny{$+1$}}
\put(-60,7){\tiny{$+1$}}
\put(-48,25){\tiny{$-1$}}
\put(-25,12){\tiny{$+1$}}
\put(-55,-14){\tiny{$+1$}}
\put(-15,-5){\tiny{$+1$}}
\put(3,-2){\tiny{$-1$}}
\end{picture}}

$M(\Gamma_{\mu}) = \hspace{.1in}
\left(\begin{array}{ccccc}
a&b&0&0&0\\
e& 0&i&j&0\\
d&0&h&k&0\\
0&f&0&l&m\\
0&c&0&0&g\\
\end{array}\right) \hspace{.3in}
K(\Gamma_{\mu}) = \hspace{.1in}
\left(\begin{array}{ccccc}
a&b&0&0&0\\
e& 0&-i&j&0\\
d&0&h&k&0\\
0&f&0&l&-m\\
0&c&0&0&g\\
\end{array}\right)$
\caption{A weighted bipartite graph, a Kasteleyn weighting, the weight matrix $M(\Gamma_{\mu})$ and the Kasteleyn matrix $K(\Gamma_{\mu})$ \label{kastmat}}
\end{figure}

A proof of the following result, which is due to Kasteleyn, can be found in \cite{Kuperberg:PermanentDeterminant}. 

\begin{theorem}[Kasteleyn] \label{KasteleynsTheorem}
Let  $\Gamma_{\mu}$ be a weighted bipartite plane graph. Then $$Z(\Gamma_{\epsilon \cdot \mu}) = \textup{Perm}(K(\Gamma_{\mu})) =  \pm \textup{Det}(M(\Gamma_{\mu})),$$ or equivalently
$$Z(\Gamma_{\mu}) = \textup{Perm}(M(\Gamma_{\mu})) =  \pm \textup{Det}(K(\Gamma_{\mu}))$$
\end{theorem}

\begin{remark}
It is known that Kasteleyn's theorem does not hold in general for non planar graphs. In fact, it holds if and only if the graph does not have $K_{3,3}$ as a minor \cite{LovaszPlummer:MatchingTheory}.
\end{remark}
 
\section{The Alexander polynomial}

We begin by giving a determinant definition of the Alexander polynomial due to Fox which can be found in \cite{CrowellFox:KnotTheory}. By using Kasteleyn's theorem we will construct a bipartite plane graph such that the partition function of this graph is the Alexander polynomial.
Finally, we will show that this approach yields Kauffman's state sum model \cite{Kauffman:OnKnots}.  

While it might initially seem strange to express a determinant by a partition function this combinatorial model proved to be useful for example in the study of Ozsv\'ath-Szab\'o-Knot-Floer homology theory, see e.g. \cite{OzsvathSzabo:Alternating,Lowrance:KnotFloerWidth} and compare with \cite{DasbachLowrance:KnotSignature} .

Moreover, some properties of the Alexander polynomial follow directly from this approach, e.g. it is an easy exercise to show that the Alexander polynomial of an alternating knot has coefficients of alternating signs.

\subsection{The Alexander polynomial as a determinant}\label{Adet}

Consider a knot $K\subset S^3$ along with some fixed generic diagram $D_K$. Label the faces of $D_K$ with $a_0, \ldots, a_m$ where $a_0$ is the unbounded face. Choose some base point above the plane of projection. Let $A_0, \ldots, A_m$ be loops in $\pi :=\pi_1(S^3-K)$ given by passing through face $a_i$ and returning through $a_0$ to the base point. Thus, loop $A_0$ is trivial in $\pi$. Using this notation we have the  Dehn presentation for the knot group $$ \pi= \left< A_0, \ldots, A_m| r_1, \ldots r_{m-1}, A_0 \right>$$ where the $r_i$ are relations coming from the crossings in $D_K$. In particular the Dehn relation for a crossing shown in Figure \ref{Dehncross} is $r: A_4A_2^{-1} = A_3A_1^{-1}$.
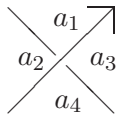
\begin{figure}[ht]
\begin{picture}(40,40)(0,-40)
\put(40,0){\line(-1,-1){40}}
\put(0,0){\line(1,-1){18}}
\put(40,-40){\line(-1,1){18}}
%\put(0,0){\line(1,0){10}}
%\put(0,0){\line(0,-1){10}}
\put(40,0){\line(-1,0){10}}
\put(40,0){\line(0,-1){10}}
\put(17,-7){$a_1$}
\put(0,-22){ $a_2$}
\put(31,-22){$a_3$}
\put(14,-38){ $a_4$}
\end{picture}
\caption{A crossing with labeled faces.} \label{Dehncross}
\end{figure}

Setting $A_0 = 1$ and incorporating that change in the relations $r_i$ we get a presentation of $\pi$ with $m$ generators and $m-1$ relations $$\pi = <A_1, \ldots, A_m: r_1, \ldots, r_{m-1}>.$$

Let $\mathcal{F}_m$ be the free group on $m$ generators $A_1, \ldots, A_m$. Then the  free derivative$\frac{\partial}{\partial A_i}$ is a map from $\mathcal{F}_m$  to $\mathbb{Z}[\mathcal{F}_m]$ recursively defined by
\begin{itemize}
\item{$\frac{\partial}{\partial A_i}\left(1\right) = 0$,}
\item{$\frac{\partial}{\partial A_i}\left(A_j\right) = \delta_{ij}$,}
\item{$\frac{\partial}{\partial A_i}\left(-A_j\right) = -\delta_{ij}A_j^{-1}$,}
\item{and $\frac{\partial}{\partial A_i}\left(wA_j\right) = \frac{\partial}{\partial A_i}\left(w\right) + w\frac{\partial}{\partial A_i} \left( A_j\right)$ for any word $w\in \mathcal{F}_m$.}
\end{itemize}

Consider the map $\phi: \mathcal{F}_m \rightarrow \pi$ defined by the map $\phi(A_i) = A_i$. We can extend this to a map $\phi: \mathbb{Z}[\mathcal{F}_m] \rightarrow \mathbb{Z}[\pi]$, and the kernel of this map will be generated by the relations $r_i$ in $\mathcal{F}_m$.  Let $\psi: \mathbb{Z}[\pi] \rightarrow \mathbb{Z}[t^{\pm 1}]$ be the abelianization mapping which will take meridians positively linking the knot to the variable $t$.

Let $M_K$ be the $(m-1)\times m$ dimensional matrix with $ij^{th}$ entry given by $ \psi \circ \phi\left( \frac{\partial r_i}{\partial A_j} \right)$. Remove any column corresponding to a face of the diagram $D_K$ that is adjacent to the unbounded face obtaining a square matrix $M_K'$. Up to sign and multiplication by a power of $t^{\pm 1}$, the determinant of $M_K'$ is independent of the choice of the adjacent face and it is invariant under Reidemeister moves.  Thus we define the Alexander polynomial $K$, denoted $\Delta_K(t)$, to be
$$\Delta_K(t) \dot{=} \textup{det}(M_K'),$$
where $\dot{=}$ means equality up to multiplication with $\pm t^{k}$ for some power $k$.

It will turn out that the entries of the matrix $M_K$ are all either  0 or $\pm 1$ or $\pm t$.
More specifically,  the free derivatives of all the relations $r_i$ will end up being $0$ or $1$ or meridians in $\pi$. In other words this means that while we form the matrix $M_K$ using the Dehn presentation for the knot group, the free derivatives of the relations are up to a sign the so called Wirtinger generators.  
Since it will become important in Section \ref{TwistedAlex} we recall the Wirtinger presentation for $\pi$ and then show the calculation of the free derivatives. 

Given the diagram $D_K$ we can label the arcs of the knot $c_1, \ldots, c_{\ell}$. Again choose some base point above the plane of projection. For $1\leq i\leq \ell$ let the loop $x_i$ be the meridian that leaves the base point, positively links the arc $c_i$, and returns to the base point. We again get a relation  $r'_{j}$ at each crossing. Using these generators and relations we get another presentation for $\pi$ known as the Wirtinger presentation.
$$\pi = <x_1, \ldots, x_{\ell}| r_1', \ldots, r_{\ell-1}'>$$
Under abelianization we see that $\psi(x_i) = t$ for all $i$. 
\begin{lem}\label{DW}
The free derivatives of Dehn relations are either $0$, $\pm 1$ or, up to a sign, Wirtinger generators.
Thus the matrix $M_K'$ has entries either $0$, $\pm 1$ or $\pm t$.
By multiplying suitable rows and colums with $-1$ we can assume that $M_K'$ has only non-negative entries $0$, $1$ or $t$.
\end{lem}
\begin{proof}
As we noted above the Dehn relation coming from this crossing in Figure \ref{Dehncross} has the form $ A_4A_2^{-1} = A_3A_1^{-1}$. Furthermore if we say that $x$ is the Wirtinger generator assigned to the overcrossing strand in Figure \ref{Dehncross} then we have that $A_4A_2^{-1} = x$.

Now we take the free derivatives of the relation $r$ with respect to each variable and see that the results can be written completely in terms of $x$ and $1$. Indeed, we get the following.
\begin{center}
\begin{tabular}{ll}
$\frac{\partial}{\partial A_1}(r) = A_4A_2^{-1} = x$& $ \hspace{.5in}\frac{\partial}{\partial A_2}(r) = -A_4A_2^{-1} = -x$\\
$\frac{\partial}{\partial A_3}(r) = -1$&$ \hspace{.5in}\frac{\partial}{\partial A_4} (r) = 1$\\
\end{tabular}
\end{center}

It remains to show that the matrix can be transformed into a matrix with only non-negative entries by multiplying suitable rows or columns by $-1$.
For that color the faces of the diagram black/white so that no two adjacent faces have the same color. We see that locally the partial derivatives are negative
at the two generators corresponding to either the black faces or the white faces. By multiplying all entries locally by $-1$, i.e. multiplying a row by $-1$. we can assume that the partial derivatives are negative at, say, the black faces.  By multiplying all columns corresponding to black faces we obtain the result.
\end{proof}

\subsection{The dimer state sum}

Using Kasteleyn's Theorem \ref{KasteleynsTheorem} we get the following construction which expresses the Alexander polynomial, given as a determinant, as a partition function of a certain bipartite graph.

For this we take a knot diagram and chose two adjacent faces that we disregard.
As in Lemma \ref{DW} the Alexander polynomial is the determinant of a matrix $M_K'$ that is indexed by the remaining faces and the relations, which correspond to the crossings of the diagram. For each relation (i.e. crossing) locally the entries in the matrix are described by the picture in Figure \ref{LocalWeights} as proven in Lemma \ref{DW}.

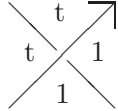
\begin{figure}[ht]
\scalebox{1}{\begin{picture}(40,40)(0,-40)
\put(40,0){\line(-1,-1){40}}
\put(0,0){\line(1,-1){18}}
\put(40,-40){\line(-1,1){18}}
%\put(0,0){\line(1,0){10}}
%\put(0,0){\line(0,-1){10}}
\put(40,0){\line(-1,0){10}}
\put(40,0){\line(0,-1){10}}
\put(17,-7){t}
\put(2,-22){ t}
\put(31,-22){1}
\put(14,-38){ 1}
\end{picture}}
 \caption{Local weights at a crossing \label{LocalWeights}}
\end{figure}

Thus we obtain the following:
  
\begin{alg}{{\bf The dimer state sum model}} \label{dimstat}
\begin{enumerate}
\item[(D1)]{Construct a bipartite plane graph $\Gamma = (V_1, V_2, E)$ as follows.
\begin{itemize}
\item{The vertex set $V_1$ is the set of crossings of the diagram.}
\item{The vertex set $V_2$ is the set of faces of the diagram.}
\item{Given vertices $x\in V_1$ and $y\in V_2$ the edge $(x,y)$ is in the set $E$ if and only if the crossing $x$ is incident on the face $y$.}
\end{itemize}
We will call this the Alexander graph.}
\item[(D2)]{Use the weighting system of \ref{LocalWeights} to define a weighting 
$\alpha: E \rightarrow \mathbb{C}[t]$ on $\Gamma$.}
\item[(D3)] Choose a Kasteleyn weighting. The next paragraph will describe a particular weighting due to Kauffman.
\item[(D3)]{Calculate the partition function $Z(\Gamma_{\alpha}) = \sum_{m\in \mathcal{M}} \prod_{e\in m} \alpha(e)$.}
\end{enumerate}
\end{alg}

\subsubsection{Kauffman's Kasteleyn weighting}

The following proposition describes a way to choose a Kasteleyn weighting which is due to Kauffman. The proof that it gives a Kasteleyn weighting is an easy exercise.

\begin{pro}
The assignments of weights given by Figure \ref{KauffmanTrick} is a Kasteleyn weighting,

\begin{figure}[ht]
\scalebox{1}{ 
\begin{picture}(40,40)(0,-40)
\put(40,0){\line(-1,-1){40}}
\put(0,0){\line(1,-1){40}}
\put(0,0){\line(1,0){10}}
\put(0,0){\line(0,-1){10}}
\put(40,0){\line(-1,0){10}}
\put(40,0){\line(0,-1){10}}
\put(15,-10){-1}
\put(-8,-22){ +1}
\put(24,-22){ +1}
\put(12,-36){+1}
\end{picture}
}
\caption{Kauffman's Kasteleyn weighting \label{KauffmanTrick}}
\end{figure}
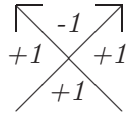

\end{pro}

Now the following Theorem immediately follows by construction and Kasteleyn's theorem \ref{KasteleynsTheorem}.

\begin{pro}
The dimer state sum model described in Algorithm \ref{dimstat} calculates the Alexander polynomial. In other words for the weighted graph $\Gamma_{\alpha}$ we have $\Delta_L(t) \dot{=}  Z(\Gamma_{\alpha})$.
\end{pro}

\begin{example}
Consider the trefoil given in Figure \ref{trefoilex}. 

\begin{figure}[ht]
\begin{picture}(100,70)(-20,0)
\put(0,0){\includegraphics[width=.8in]{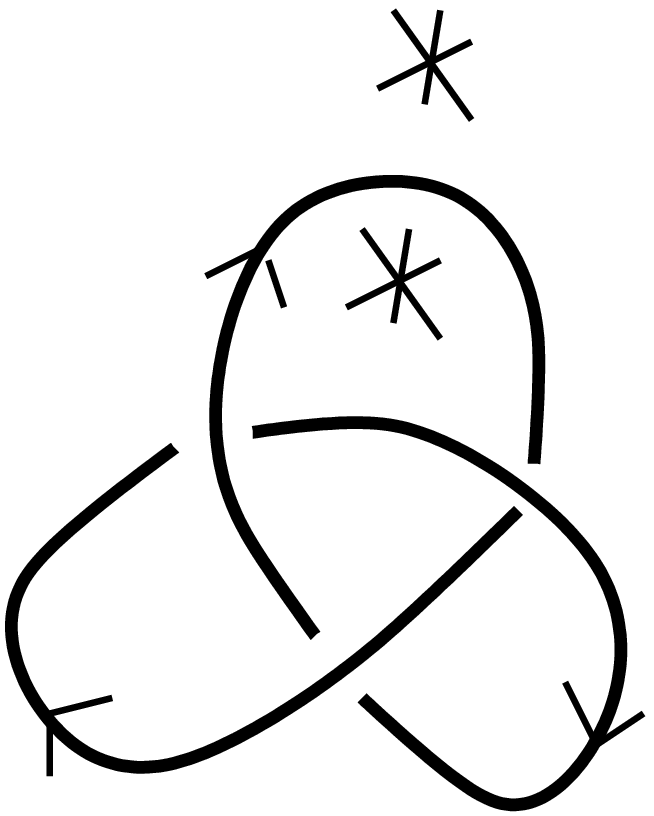}}
\put(16,24){\tiny{t}}
\put(22,27){\tiny{1}}
\put(35,10){\tiny{t}}
\put(20,11){\tiny{1}}
\put(29,18){\tiny{1}}
\put(38,25){\tiny{1}}
\put(44,20){\tiny{1}}
\end{picture}
\begin{picture}(100,70)(-20,0)
\put(0,0){\includegraphics[width=.8in]{FaceSelect.eps}}
\put(16,24){\tiny{1}}
\put(22,27){\tiny{1}}
\put(35,10){\tiny{1}}
\put(20,11){\tiny{-1}}
\put(29,18){\tiny{1}}
\put(38,25){\tiny{1}}
\put(44,20){\tiny{-1}}
\end{picture}

\caption{Local weights and Kauffman's Kasteleyn weights for the trefoil \label{trefoilex}}
\end{figure}

Associated to that diagram we have the following bipartite plane graph where the black vertices correspond to the faces and the white vertices to the crossings of the diagram.
This graph has three dimer coverings.

\begin{figure}[ht]
\raisebox{-20pt}{\scalebox{.8}{\begin{picture}(30,80)(-30,0)
\put(0,0){\includegraphics[width=1.5in]{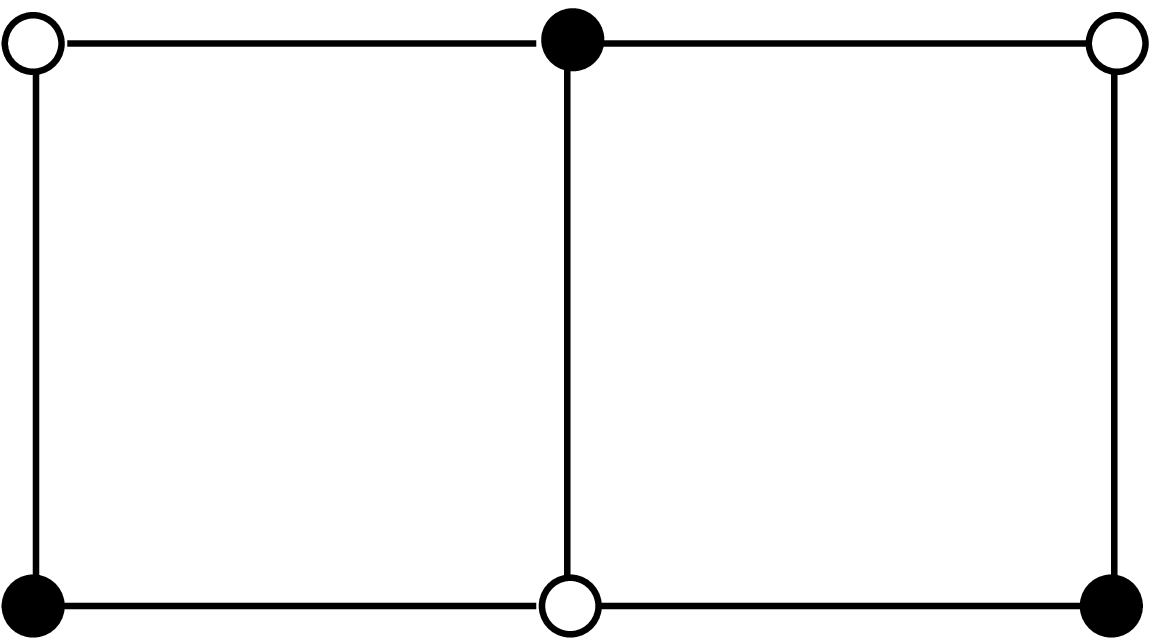}}
\put(-5,25){t}
\put(42,25){1}
\put(90,25){-1}
\put(25,60){1}
\put(75,60){1}
\put(25,-10){-1}
\put(75,-10){t}
\end{picture}}}
\hspace{2in} $\mathcal{M} = \left\{ \raisebox{-10pt}{\includegraphics[width=.7in]{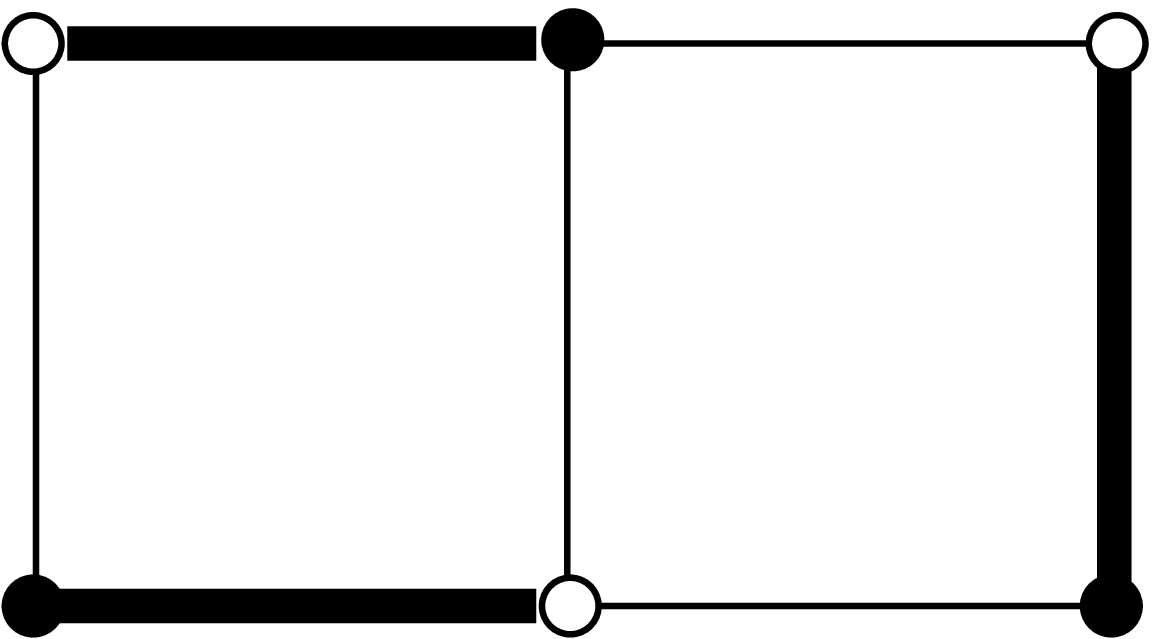}},\hspace{.3in}\raisebox{-10pt}{\includegraphics[width=.7in]{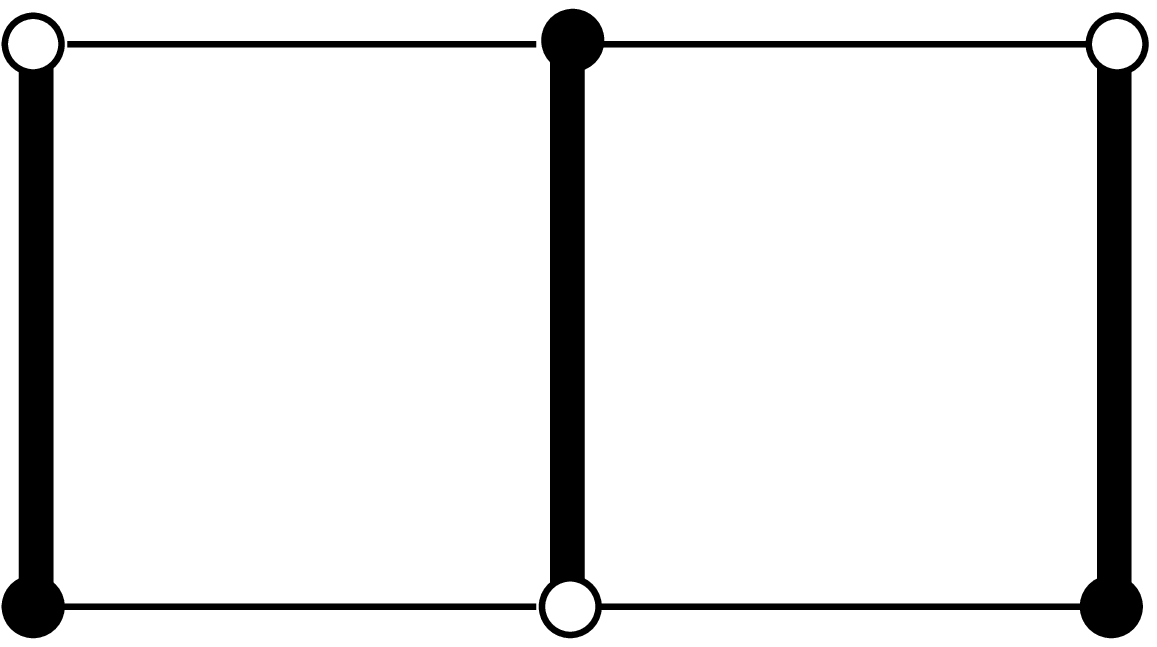}},\hspace{.3in}\raisebox{-10pt}{\includegraphics[width=.7in]{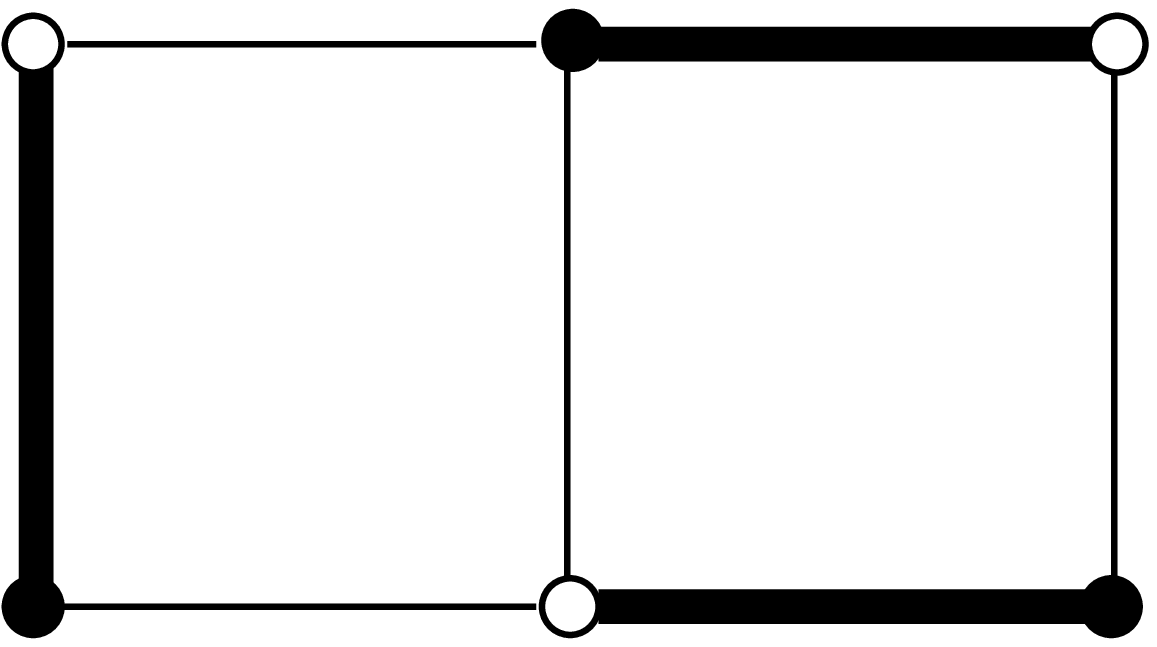}} \right\}$
\end{figure}

\noindent Then we calculate the Alexander polynomial using Algorithm \ref{dimstat} as follows. $$ Z(\Gamma_{\alpha}) = \alpha \left( \raisebox{-12pt}{\includegraphics[width=.8in]{pm3.eps}}\right) +  \alpha \left( \raisebox{-12pt}{\includegraphics[width=.8in]{pm2.eps}}\right) + \alpha \left( \raisebox{-12pt}{\includegraphics[width=.8in]{pm1.eps}}\right) = t^2 -t + 1$$
\end{example}

\subsection{Kauffman's state sum}
It turns out that with  the dimer state sum model described in the previous section we recover the state sum model by Kauffman for the Alexander polynomial.
We will briefly describe this model here. 

Again let $D_K$ be a generic diagram of a knot $K$. Choose the unbounded face of $D_K$ and one other face that is adjacent to the unbounded face. We will disregard these faces in our calculation. A simple Euler characteristic argument tells us that the number of crossings of the diagram is equal to the number of faces remaining.  Say this number is $m$.

\begin{alg}{{\bf Kauffman state sum model}}\label{KA}
%Begin with a diagram $D_K$ for a knot $K$ for which two faces have been disregarded in the manner described above.
\begin{enumerate}
\item[(K1)]Decorate the diagram $D_K$ with the product of the two weights around each crossing that are depicted in Figure \ref{LocalWeights} and Figure \ref{KauffmanTrick}.
\item[(K2)]{Find all possible ways to distribute $m$ markers on the diagram so that each remaining face and each crossing has exactly one marker. Each of these configurations is called a state. Denote the set of all states by $\mathcal{S}$.}
\item[(K3)]{For each $s\in\mathcal{S}$ let $w(s)$ be the product of the weights associated to the state.  Then $$\Delta_K(t) \dot{=} \sum_{s\in \mathcal{S}} w(s).$$} 
\end{enumerate}
\end{alg}

The weights given in the lefthand diagram in step $K1$ of Algorithm \ref{KA} are encoding the free derivatives described in the previous section. The values shown are not exactly $\psi \circ \phi\left( \frac{\partial r_i}{\partial A_j} \right)$, but simple matrix operations discussed in \cite{Kauffman:OnKnots} give us these unsigned weights which are more convenient for calculation.

\begin{example}\label{kaufex}
Consider the following weighted diagram of the trefoil and its three states. (The two starred regions are the disregarded faces.) 
\begin{figure}[ht]
\raisebox{-30pt}{\begin{picture}(100,70)(-20,0)
\put(0,0){\includegraphics[width=.8in]{FaceSelect.eps}}
\put(16,24){\tiny{t}}
\put(22,27){\tiny{1}}
\put(35,10){\tiny{t}}
\put(20,11){\tiny{-1}}
\put(29,18){\tiny{1}}
\put(38,25){\tiny{1}}
\put(44,20){\tiny{-1}}
\end{picture}}
\hspace{.5in} $\mathcal{S} = \left\{ \raisebox{-20pt}{\includegraphics[width=.6in]{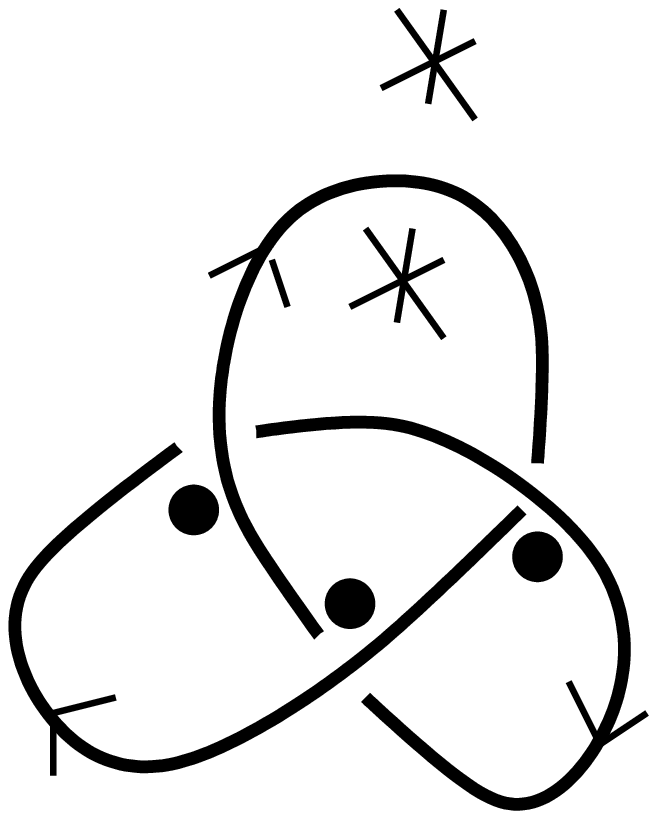}},\hspace{.3in}\raisebox{-20pt}{\includegraphics[width=.6in]{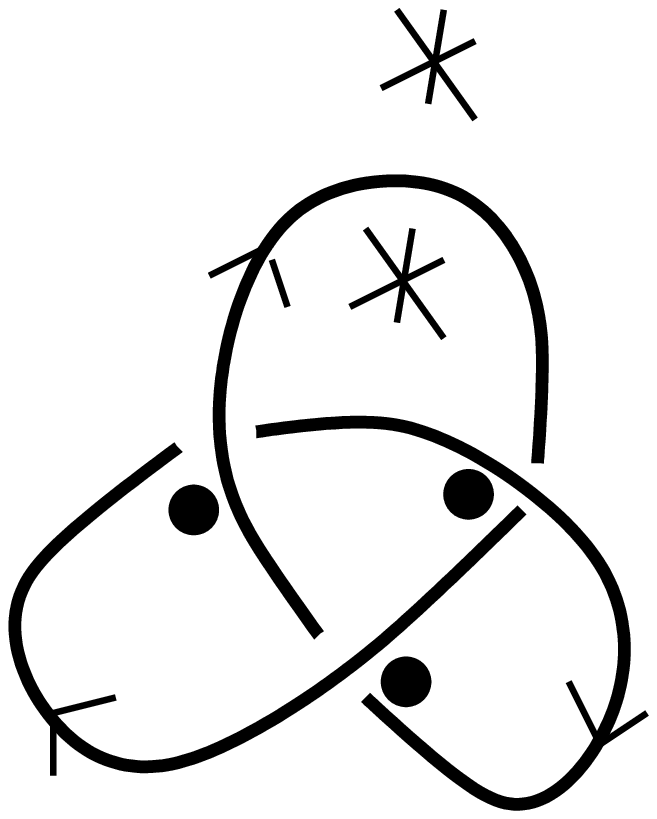}},\hspace{.3in}\raisebox{-20pt}{\includegraphics[width=.6in]{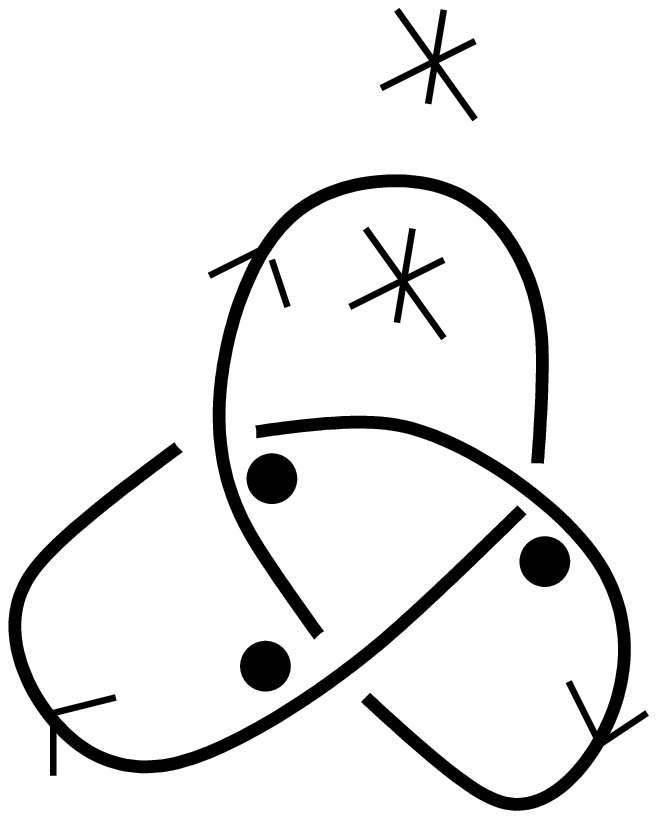}} \right\}$
\end{figure}

\noindent Then using Algorithm \ref{KA} we calculate the Alexander polynomial as follows.
\begin{figure}[ht]
\scalebox{.9}{$\Delta_{K}(t) \dot{=} w\left( \raisebox{-15pt}{\includegraphics[width=.4in]{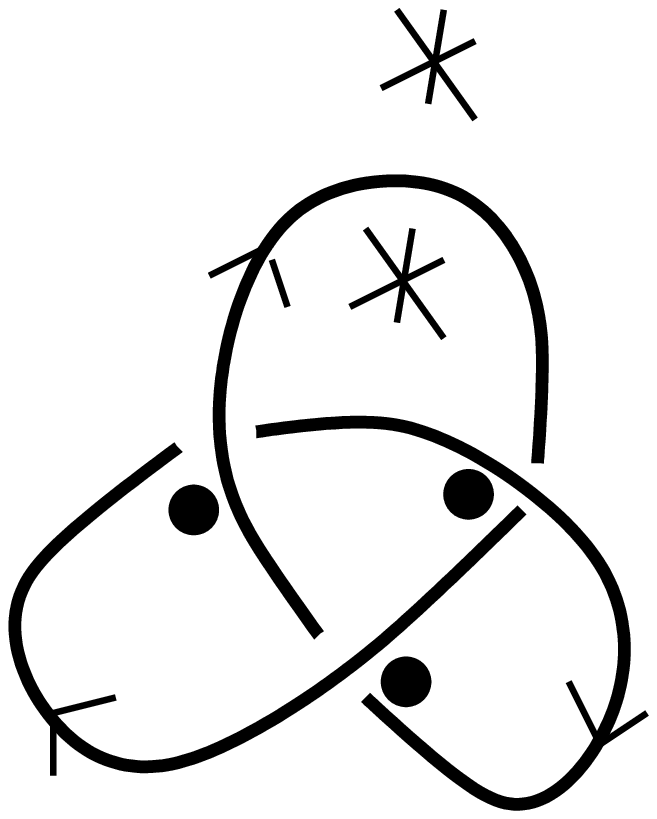}}\right) +  w\left( \raisebox{-15pt}{\includegraphics[width=.4in]{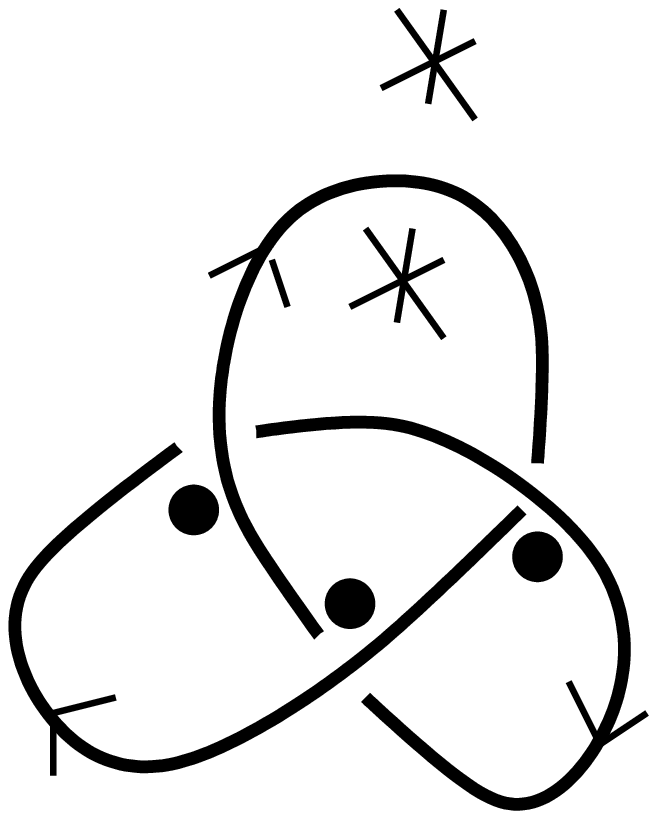}}\right) + w\left( \raisebox{-15pt}{\includegraphics[width=.4in]{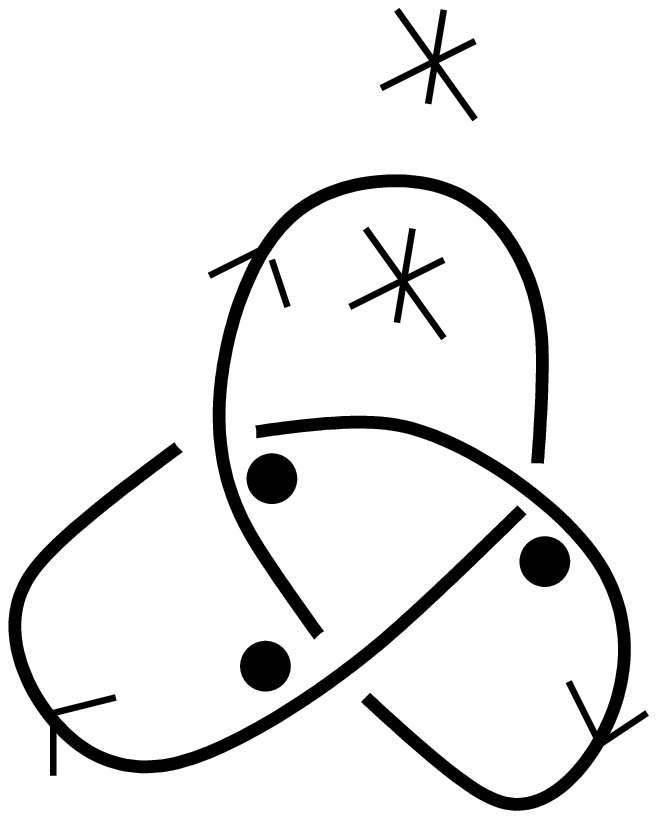}}\right) = t^2 -t + 1$}
\end{figure}
\end{example}

\section{The twisted dimer model}
We will begin with a brief description of the twisted Alexander polynomial. This polynomial was originally described by Lin using the Seifert matrix \cite{Lin:TwistedAlexander}  and has also been described by Kirk-Livingston using the language of Reidemeister torsion \cite{KirkLivingston:TwistedAlexander}.  Our description is adapted from Wada's exposition in \cite{Wada:TwistedAlexander} which defines the twisted Alexander polynomial for all finitely presented groups, We choose this description because it most clearly shows the generalization of the matrix described in Section \ref{Adet}. The polynomial that Wada defines to be the twisted Alexander is a certain quotient. On this point we depart from Wada's terminology and follow the work of Lin who refers to the numerator of Wada's quotient as the twisted Alexander polynomial.

\subsection{The twisted Alexander polynomial for knots} \label{TwistedAlex}

Recall the maps $\phi: \mathbb{Z}[\mathcal{F}_m] \rightarrow \mathbb{Z}[\pi]$ and $\psi: \mathbb{Z}[\pi] \rightarrow \mathbb{Z}[t^{\pm1}]$ from section \ref{Adet}. Let $R$ be an integral domain, and let $\rho$ be a finite dimensional representation of $\pi$, $\rho: \pi \rightarrow GL_n(R)$. We can extend $\rho$ to a ring homomorphism $\rho: \mathbb{Z}[\pi] \rightarrow M_n(R)$.

We put these maps together to define
$$\Phi = (\rho \otimes \psi) \circ \phi: \mathbb{Z}[\mathcal{F}_m] \rightarrow M_n(R[t^{\pm 1}]).$$
Now using this map construct $M_{K,\rho}$ a block matrix with $ij^{th}$ block entry $\Phi\left( \frac{\partial r_i}{\partial A_j} \right)$. Wada calls this matrix the Alexander matrix associated to the representation $\rho$. The Alexander matrix in this case has dimensions $n(m-1) \times nm$. Consider the submatrix $M'_{K,\rho}$ that comes from deleting any block column corresponding to a face of $D_K$ that is adjacent to the unbounded face. 

Up to sign and multiplication by a power of $t^{\pm 1}$ the determinant of $M_{K,\rho}'$ is well-defined and invariant under Reidemeister moves. Thus we define the twisted Alexander polynomial of the pair $K, \rho$ denoted $\Delta_{K,\rho}(t)$ to be $$\Delta_{K,\rho}(t) \dot{=} det(M'_{K,\rho}).$$

Given the trivial representation $\rho: \pi \rightarrow \mathbb{C}$ we see that $\Delta_{K,\rho}(t) \dot{=} \Delta_K(t)$. In general the rows in the twisted Alexander matrix replace occurrences of $1$ in Algorithms \ref{KA} and \ref{dimstat} with $Id\in GL_n(R)$ and occurrences of $t$ with $tX$ where $X$ is the element of $GL_n(R)$ assigned to the Wirtinger generator linking the overstrand at the associated crossing. 

We conclude this subsection with an example calculation of the twisted Alexander polynomial.
\begin{example} \label{trefex}
Consider the following diagram of the trefoil. We have labeled the regions of the diagram that give the Dehn generators as well as the arcs of the knot that give the Wirtinger generators.
\begin{figure}[ht]
\scalebox{.75}{\begin{picture}(100,60)(0,0)
\put(0,0){\includegraphics[width=1in]{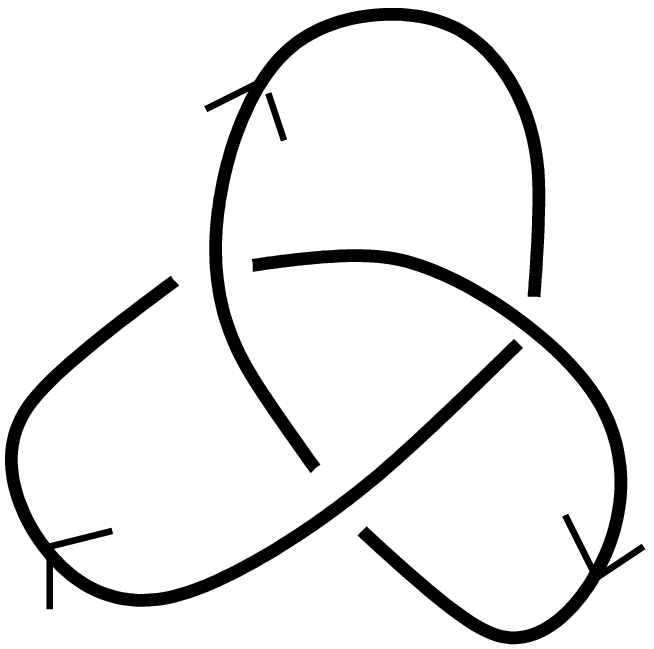}}
\put(0,50){$a_0$}
\put(10,20){$a_1$}
\put(35,30){$a_2$}
\put(37,55){$a_4$}
\put(49,15){$a_3$}
\put(63,60){$\bf{c_1}$}
\put(75,10){$\bf{c_2}$}
\put(-7,10){$\bf{c_3}$}
\end{picture}}
\end{figure}

As noted by Proposition \ref{DW} we only need to know what our representation $\rho$ does to Wirtinger generators. We consider the coloring representation $\rho: \pi \rightarrow GL_3(\mathbb{Z})$ given by $$\rho(C_1) = \left(\begin{array}{ccc}
0&1&0\\
1&0&0\\
0&0&1\\
\end{array}\right), \hspace{.5in}
\rho(C_2) = \left(\begin{array}{ccc}
1&0&0\\
0&0&1\\
0&1&0\\
\end{array}\right), \hspace{.5in} \rho(C_3) = \left(\begin{array}{ccc}
0&0&1\\
0&1&0\\
1&0&0\\
\end{array}\right)$$
We set $A_0 = 1$ in the knot group, and we have three relations remaining. 

\begin{itemize}
\item[($r_1$)]{$A_2A_1^{-1}=A_4 = C_1$}
\item[($r_2$)]{$A_2A_4^{-1}=A_3 = C_2$}
\item[($r_3$)]{$A_2A_3^{-1}=A_1 = C_3$}
\end{itemize}

Now we build a block matrix with entries $\Phi\left( \frac{\partial r_i}{\partial A_j} \right)$. The map $\Phi$ will assign identity matrices to occurrences of $1$ and the various representation matrices scaled by $t$ when Wirtinger generators occur. In this case the matrix we get is as follows.

$$\left( \begin{array}{ccc|ccc|ccc|ccc}
0&-t&0&1&0&0&0&0&0&-1&0&0\\
-t&0&0&0&1&0&0&0&0&0&-1&0\\
0&0&-t&0&0&1&0&0&0&0&0&-1\\
\hline 0&0&0& 1&0&0& -1&0&0& -t&0&0\\
0&0&0&0&1&0&0&-1&0&0&0&-t\\
0&0&0&0&0&1&0&0&-1&0&-t&0\\
\hline -1&0&0&1&0&0&0&0&-t& 0&0&0\\
0&-1&0&0&1&0&0&-t&0&0&0&0\\
0&0&-1&0&0&1&-t&0&0&0&0&0\\
\end{array}\right)$$

Finally we remove the last block column corresponding to face $a_4$, and we take the determinant. We get $\Delta_{\textup{Tref}, \rho}(t) = -t^6 + t^5+t^4-2t^3+t^2+t-1 = -(-1+t)^2(1+t)^2(1-t+t^2)$.
\end{example}

\subsection{The twisted Alexander graph}
We want to build a graph that encodes the Alexander matrix for a pair $K,\rho$. This graph, which we will call the twisted Alexander graph, is a generalization of the Alexander graph defined in the dimer state sum algorithm. The twisted Alexander graph replaces single edges in the Alexander graph by ``twisted edges". For an $n$-dimensional representation these ``twisted edges" are each a copy of $K_{n,n}$ (the complete bipartite graph on $2n$ vertices) that will eventually encode the associated block entry in the Alexander matrix. 

Begin by fixing a knot $K$ with diagram $D_K$ and representation $\rho: \pi \rightarrow GL_n(R)$. As with the original Alexander graph, choose two adjacent faces of the diagram to disregard, one of which is the unbounded face. 

\begin{df}\label{twal}
Construct a bipartite graph $\Gamma' = (V'_1, V'_2, E')$ as follows.
\begin{itemize}
\item{The vertex set $V'_1$ has $n$ vertices for each crossing of the diagram.}
\item{The vertex set $V'_2$ has $n$ vertices for each face of the diagram.}
\item{If a face and a crossing are incident, insert a copy of $K_{n,n}$ going between the vertices corresponding to that face and crossing. If the face and crossing are not incident, no edges should connect their corresponding vertex sets.}
\end{itemize}
We will call this the {\bf twisted Alexander graph}.
\end{df}

\begin{df}\label{encode}
Let $M\in M_n(R([t^{\pm1}]))$. The complete bipartite graph $K = K_{n,n}$ has 
all possible edges between two sets of $V_r$ and $V_c$ each consisting of $n$ vertices. Let $K_M$ be the complete bipartite graph weighted according to $M$. More precisely begin by enumerating the vertices in $V_r$ with $v_{1,r}, \ldots, v_{n,r}$ and the vertices in $V_c$ with $v_{1,c}, \ldots, v_{n,c}$. Now the edge between $v_{i,r}$ and $v_{j,c}$ gets the entry in the $ij^{th}$ position of $M$. If the entry in $M$ is zero, we do not include the edge. We will call the weighted graph $K_M$ the {\bf graph encoding $M$}. 
\end{df}

\begin{example}
Let
$M = \left(\begin{array}{ccc}
a&b&0\\
0&0&c\\
d&0 & e\\
\end{array}\right)$. Then the graph encoding $M$ is
\begin{figure}[ht]
\scalebox{.7}{\begin{picture}(100,50)(-20,10)
\put(0,0){\includegraphics[width=1in]{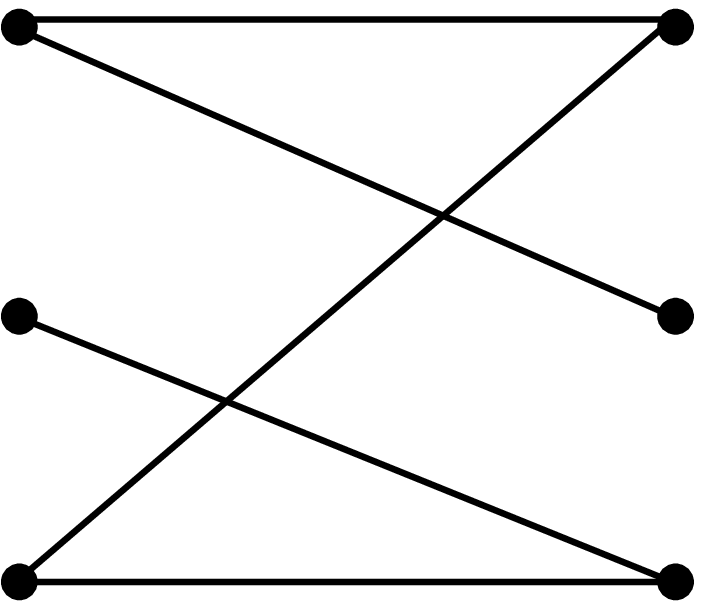}}
\put(30,65){$a$}
\put(8,46){$b$}
\put(60,10){$c$}
\put(25,30){$d$}
\put(30,5){$e$}
\end{picture}}
\end{figure}
\end{example}

Enumerate the vertices at every crossing and every face of the twisted Alexander graph $\Gamma'$. We endow the graph $\Gamma'$ with the weighting $\alpha_{\rho}: E' \rightarrow R[t^{\pm1}]$ so that the copy of $K_{n,n}$ (or the ``twisted edge")  connecting the collection of vertices for a crossing and face is the graph encoding the matrix shown in Figure \ref{twistweight}. Here $X$ is the element of $GL_n(R)$ assigned by $\rho$ to the Wirtinger generator corresponding to the overstrand in the figure.

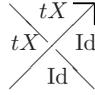
\begin{figure}[ht]
\scalebox{.8}{\begin{picture}(40,40)(0,-40)
\put(40,0){\line(-1,-1){40}}
\put(0,0){\line(1,-1){18}}
\put(40,-40){\line(-1,1){18}}
%\put(0,0){\line(1,0){10}}
%\put(0,0){\line(0,-1){10}}
\put(40,0){\line(-1,0){10}}
\put(40,0){\line(0,-1){10}}
\put(14,-7){$tX$}
\put(-3,-22){ $tX$}
\put(31,-22){Id}
\put(14,-38){ Id}
\end{picture}}
\caption{Weights for the twisted Alexander graph} \label{twistweight}
\end{figure} 

Then the following is an immediate consequence of our definitions for $\Gamma'$ and $\alpha_{\rho}$.
\begin{pro}
The permanent of the matrix $M'_{K,\rho}$ is equal to the partition function of the graph $\Gamma'$ weighted by $\alpha_{\rho}$. In other words $Z(\Gamma'_{\alpha_{\rho}}) = \textup{Perm}(M'_{K,\rho})$.
\end{pro}

\subsection{Kuperberg's tricks}
As we discussed at length in Section \ref{Dimer} we would like to find a modification of the weighting $\alpha_{\rho}$ that would allow us to encode the twisted Alexander polynomial directly. The problem is that by replacing single edges in the Alexander graph by twisted edges, we no longer necessarily have a plane graph. 

We will use two techniques due to Kuperberg to solve this problem \cite{Kuperberg:PermanentDeterminant}: edge tripling and butterflies. This will enable us to modify an embedding of the graph $\Gamma'$ and the weighting $\alpha_{\rho}$ to get a weighted plane graph that will encode the twisted Alexander polynomial as desired. For the remainder of this section, fix an embedding of $\Gamma'$.

In order to force planarity we need to require that each pair of edges in $\Gamma'$ intersect at most once. We can accomplish this by repeatedly tripling edges as shown in Figure \ref{triple}. If our weight function $\alpha_{\rho}$ assigns $a\in R[t^{\pm1}]$ to the edge we modify the weights as shown in the figure.

\begin{figure}[ht]
\begin{picture}(80,10)(20,0)
\put(0,0){\includegraphics[width=1.5in]{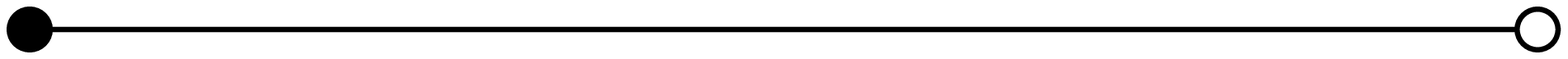}}
\put(50,5){$a$}
\end{picture}

\vspace{.35in}

\begin{picture}(80,10)(20,0)
\put(0,0){\includegraphics[width=1.5in]{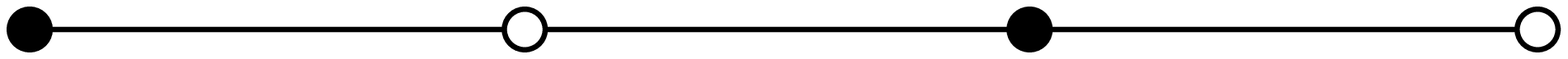}}
\put(20,5){$a$}
\put(47,5){$-1$}
\put(85,5){$1$}
\end{picture}
\caption{Replace a single edge with three edges.}\label{triple}
\end{figure}

Call the new graph obtained after tripling all necessary edges $\Gamma''$, and call the updated weight function $\alpha_{\rho, t}$. We denote the associated matrix by $M''_{K,\rho}$. Kuperberg argues that this operation changes the determinant at most up to a sign. In other words $$\det(M''_{K,\rho} )= \pm \det(M'_{K,\rho}).$$ Indeed this is easy to see as the Alexander matrix for the pair $K,\rho$ changes as follows for each edge tripling.
\begin{figure}[ht]
\begin{picture}(40, 80)(30,-30)
\put(0,0){\begin{tiny}
$\left( \begin{array}{c|ccccccc}
a& & & & * & & &\\
\\
 \hline\\
&&&&&&&\\
&&&&&&&\\
*& &&& &&&   \\
&&&&&&&\\
&&&&&&&\\
&&&&&&&
 \end{array} \right)$
 \end{tiny}}
 \put(50,-10){\Huge{*}}
 \end{picture}
 \hspace{.75in} \raisebox{20pt}{$\stackrel{tripling}{\longrightarrow}$} \hspace{.5in}
\begin{picture}(40, 80)(10,-30)
\put(0,0){\begin{tiny}$\left( \begin{array}{cc|ccccccc}
-1&a& 0& & & \cdots  & & & 0\\
1&0&&&&*&&&\\
 \hline
0&&&&&&&\\
&&&&&&&\\
\vdots& * &&& &&&   \\
&&&&&&&&\\
0&&&&&&&&
 \end{array} \right)$
\end{tiny}}
 \put(80,-10){\Huge{*}}
 \end{picture}
\end{figure}

Now we assume that our graph $\Gamma''$ has edges which pairwise intersect at most once. The final step in forcing planarity is to replace each pair of crossing edges with what Kuperberg calls a butterfly.  This object and the necessary weight modifications are shown in Figure \ref{butterfly}.

\begin{figure}[ht]
\scalebox{.7}{\begin{picture}(50,100)(20,0)
\put(0,0){\includegraphics[width=1in]{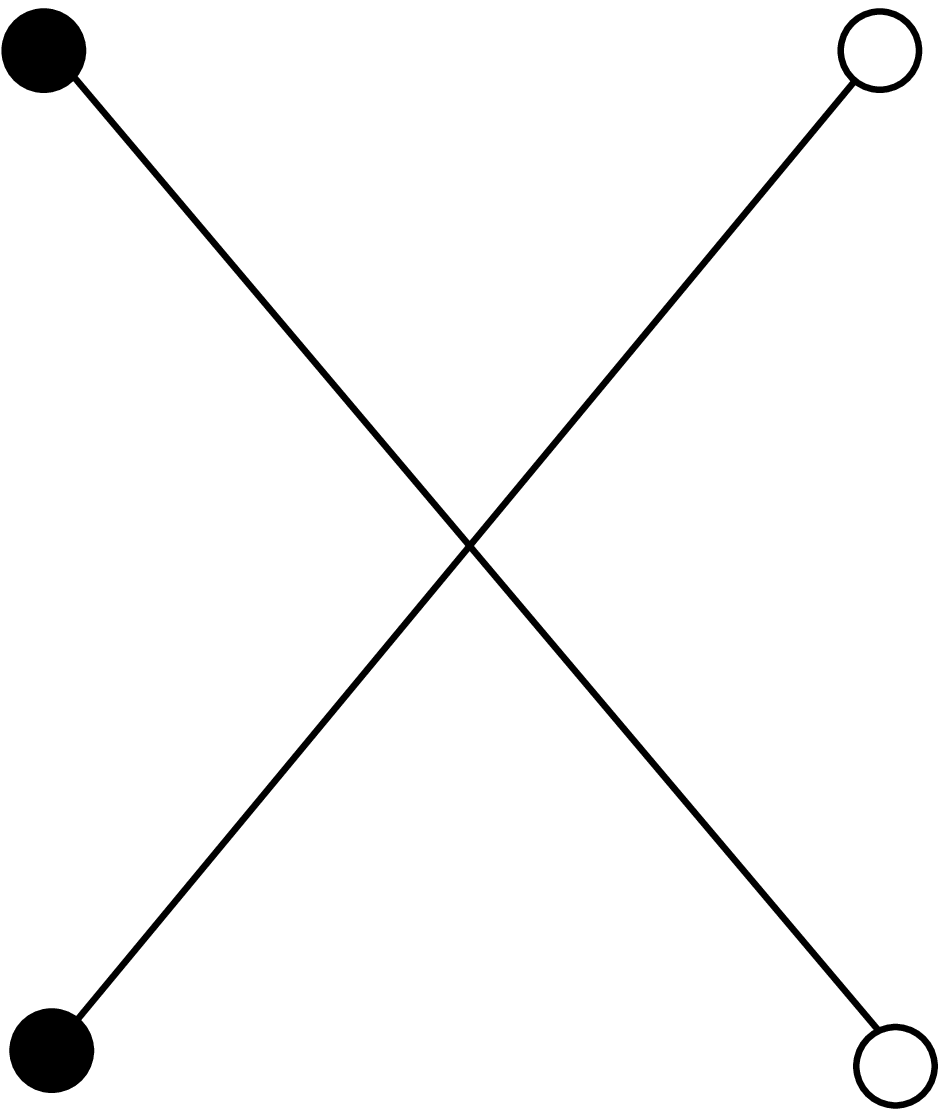}}
\put(10,50){$b$}
\put(50,50){$a$}
\end{picture}
\hspace{1in}
\begin{picture}(50,100)(0,0)
\put(0,0){\includegraphics[width=1in]{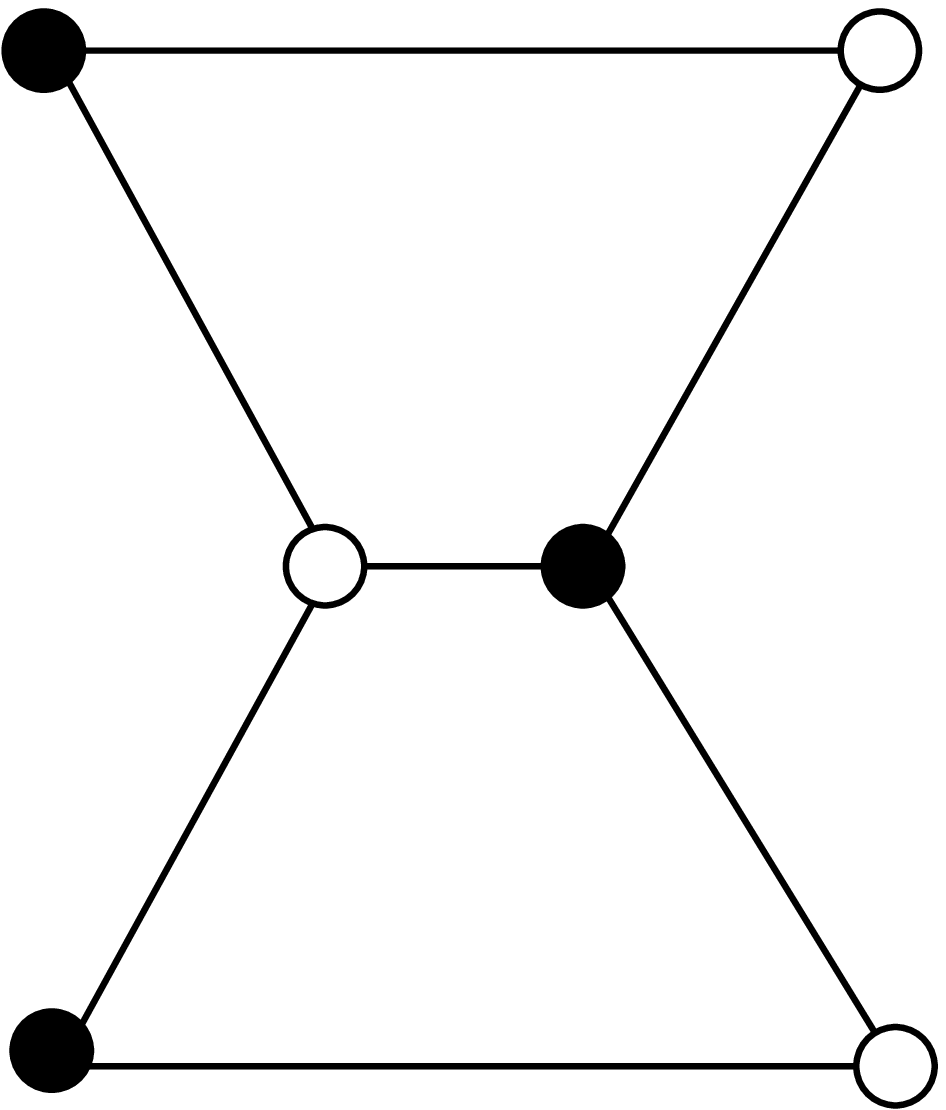}}
\put(0,50){-1}
\put(60,50){$1$}
\put(0,26){$a$}
\put(60,26){$-b$}
\put(30,85){1}
\put(30,45){-1}
\put(30,5){$ab$}
\end{picture}}
\caption{Replace crossing edges with the butterfly.}\label{butterfly}
\end{figure}

Call the new graph obtained after inserting all necessary butterflies $\Gamma'''_{\alpha_{\rho}}$, and call the updated weight function $\alpha_{\rho,b}$. 
The associated matrix is $M'''_{K,\rho}$. Again Kuperberg argues that $\det(M'''_{K,\rho} )= \pm \det(M'_{K,\rho}).$ We can also see this directly by examining the changes to the Alexander matrix for the pair $K,\rho$.

\begin{figure}[ht]
\begin{picture}(40, 80)(60,-30)
\put(0,0){\begin{tiny}$\left( \begin{array}{cc|ccccccc}
0&a& & & &   & & & \\
b&0&&&&&&&\\
 \hline
&&&&&&&\\
&&&&&&&\\
& &&& &&&   \\
&&&&&&&&\\
&&&&&&&&
 \end{array} \right)$
\end{tiny}}
 \put(63,-13){\Huge{*}}
 \put(65, 10){\Large{*}}
 \put(15, -13){\Large{*}}
 \end{picture}
 \hspace{.3in} \raisebox{25pt}{$\stackrel{butterfly}{\longrightarrow}$} \hspace{.75in}
 \begin{picture}(40, 80)(40,-30)
\put(0,0){\begin{tiny}$\left( \begin{array}{ccc|ccccccc}
-1&-1& a& 0& &   & \cdots & &&0 \\
1&1&0&&& &&&&\\
-b&0&ab&&&&&&&\\
 \hline
0&&&&&&&&\\
&&&&&&&&\\
\vdots&  &&& &&&&   \\
&&&&&&&&&\\
0&&&&&&&&&
 \end{array} \right)$
\end{tiny}}
 \put(105,-15){\Huge{*}}
  \put(105, 12){\Large{*}}
 \put(40, -20){\Large{*}}
 \end{picture}
\end{figure}

\subsection{The twisted dimer model}

We can now put together the last two subsections to state the twisted dimer model. Fix a knot $K$, a generic diagram $D_K$, and a representation $\rho: \pi \rightarrow GL_n(R)$.

\begin{alg}{\bf Twisted dimer model}
\begin{enumerate}
\item[(T1)]{Build the twisted Alexander graph as described in Definition \ref{twal} with associated weight function $\alpha_{\rho}$.}
\item[(T2)]{Choose an embedding of $\Gamma'$ that minimizes the number of edge crossings.}
\item[(T3)]{Triple any edges necessary so that each pair of edges intersects at most once, and update the weight function as shown in Figure \ref{triple}. Call the new graph $\Gamma''$ and new weight function $\alpha_{\rho,t}$.}
\item[(T4)]{Replace any crossing pairs of edges with butterflies, and update the weight function as shown in Figure \ref{butterfly}. Call the resulting graph $\Gamma'''$ and the new weight function $\alpha_{\rho,b}$.}
\item[(T5)]{Use the algorithm described in the proof of Proposition \ref{Kweight} to get a Kasteleyn weighting $\epsilon$ for $\Gamma'''$.}
\item[(T6)]{Calculate the partition function 
$$Z(\Gamma'''_{\epsilon \cdot \alpha_{\rho,b}}) =  \sum_{m\in \mathcal{M}} \left( \prod_{e\in m} \epsilon(e)\cdot \alpha_{\rho,b}(e) \right).$$}
\end{enumerate}
\end{alg}

Then our main theorem follows by construction.

\begin{theorem}
The twisted dimer model described above calculates the twisted Alexander polynomial for the pair $K, \rho$. In other words $$Z(\Gamma'''_{\epsilon \cdot \alpha_{\rho,b}}) \dot{=} \Delta_{K,\rho}(t).$$
\end{theorem}

We conclude by applying the twisted dimer model to the pair Tref, $\rho$ from Example \ref{trefex}.

\begin{example}\label{twdimex}
Recall that we are considering the following diagram for the trefoil and the representation $\rho: \pi \rightarrow GL_3(\mathbb{Z})$ that comes from a non-trivial 3-coloring. We've labeled the diagram to indicate the weight function $\alpha_{\rho}$ that will be assigned to the twisted Alexander graph.
\begin{figure}[ht]
\scalebox{1.2}{\begin{picture}(100,70)(-20,0)
\put(0,0){\includegraphics[width=.8in]{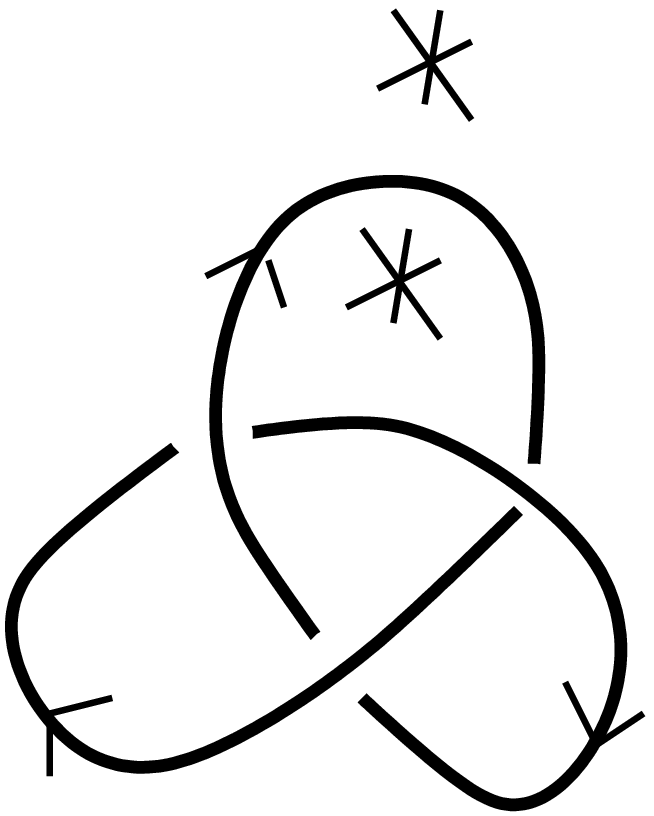}}
\put(12,24){\tiny{tX}}
\put(23,27){\tiny{Id}}
\put(35,10){\tiny{tZ}}
\put(18,12){\tiny{Id}}
\put(27,20){\tiny{Id}}
\put(35,25){\tiny{Id}}
\put(42,18){\tiny{Id}}
\end{picture}}
\end{figure}

Recall that the representation matrices are $$\rho(c_1) = X = \left(\begin{array}{ccc}
0&1&0\\
1&0&0\\
0&0&1\\
\end{array}\right)$$
$$\rho(c_2) = Y = \left(\begin{array}{ccc}
1&0&0\\
0&0&1\\
0&1&0\\
\end{array}\right) \hspace{1in} \rho(c_3) = Z = \left(\begin{array}{ccc}
0&0&1\\
0&1&0\\
1&0&0\\
\end{array}\right)$$

The twisted Alexander graph will have 3 times the number of vertices in the original Alexander graph. Each edge in the original graph will be replaced with a copy of $K_{3,3}$. The following labeling of the Alexander graph helps us to see what the twisted graph will look like.
\begin{figure}[ht]
\scalebox{.8}{\begin{picture}(140,60)(-15,0)
\put(0,0){\includegraphics[width=1.5in]{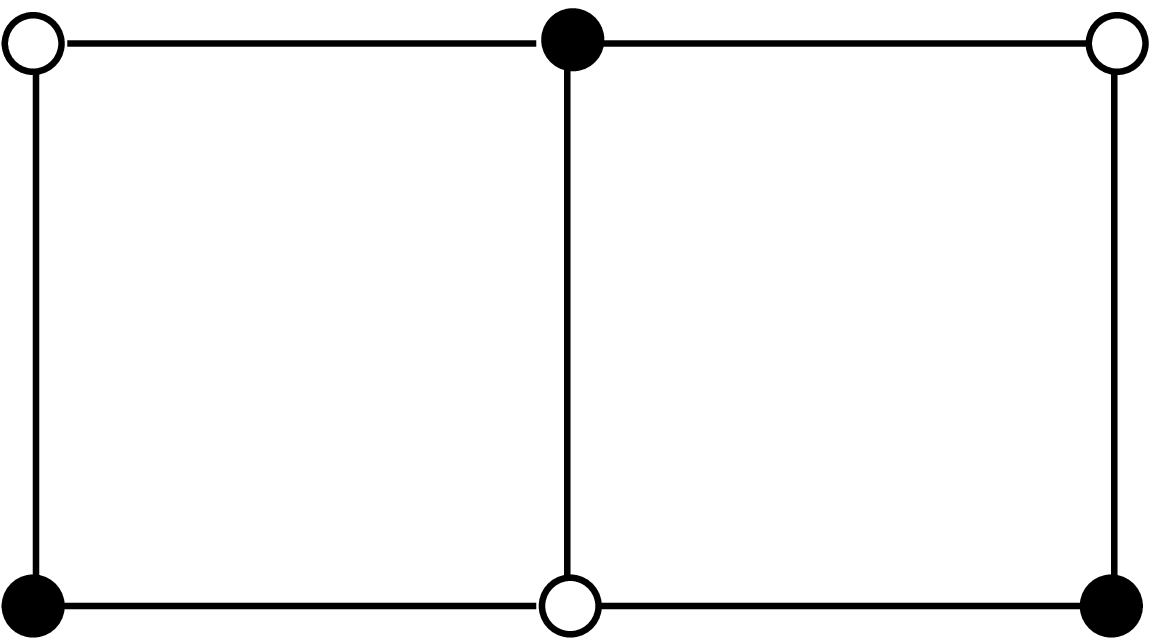}}
\put(-15,25){$tX$}
\put(42,25){Id}
\put(90,25){Id}
\put(25,60){Id}
\put(75,60){Id}
\put(25,-10){Id}
\put(75,-10){$tZ$}
\end{picture}}
\end{figure}

Finally, then, we see that the twisted Alexander graph for the pair Tref, $\rho$ has the following form. To simplify the pictures, all thickened edges have weight $t$, and all other edges have weight $1$.  As we mentioned in Definition \ref{encode}, we do not draw weight $0$ edges since any perfect matching that includes a weight $0$ edge will not contribute to the value of the partition function.
\begin{figure}[ht]
\includegraphics[width=1.8in]{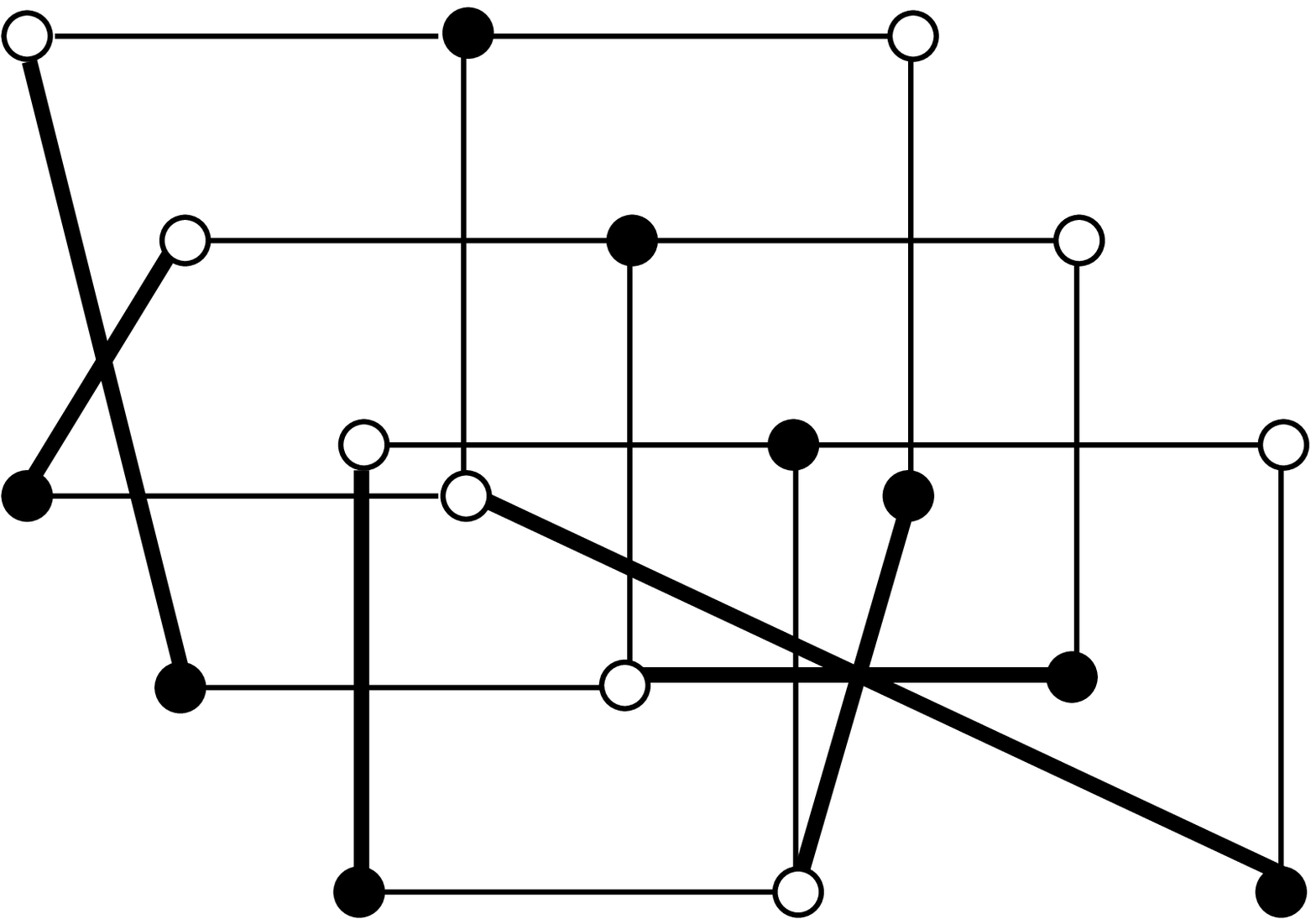}
\end{figure}

We can see through the following sequence of manipulations in Figure \ref{unravel} that, in this case, the twisted Alexander graph is a plane graph. While the embedding has changed, on the level of abstract graphs we have that $\Gamma' = \Gamma'' = \Gamma'''$ and $\alpha_{\rho} = \alpha_{\rho, t} = \alpha_{\rho,b}$.

\begin{figure}[ht]

(1) \raisebox{-30pt}{\includegraphics[width=1.3in]{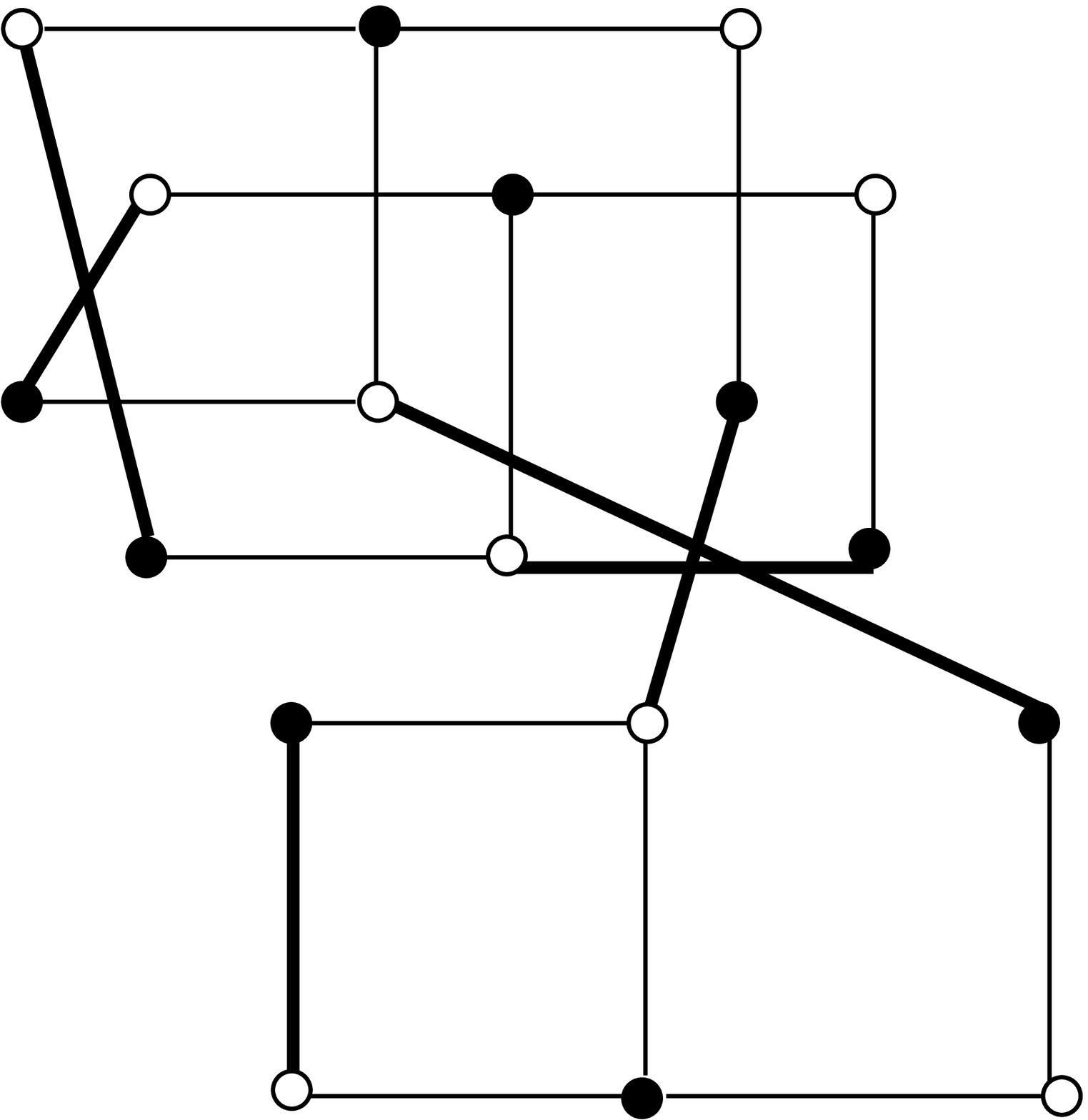}} \hspace{1in}
(2) \raisebox{-30pt}{\includegraphics[width=1.8in]{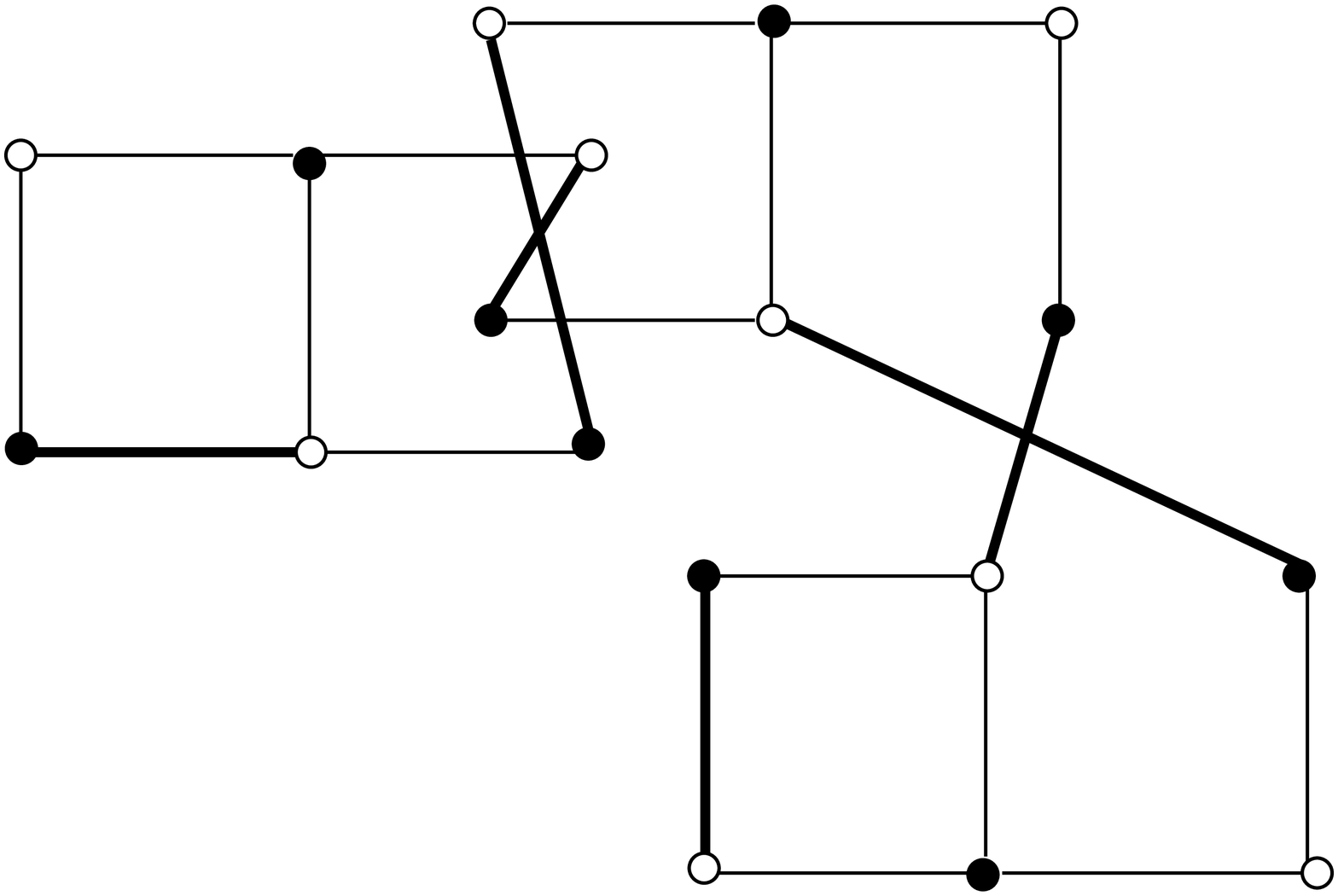}}

(3) \raisebox{-30pt}{\includegraphics[width=2.2in]{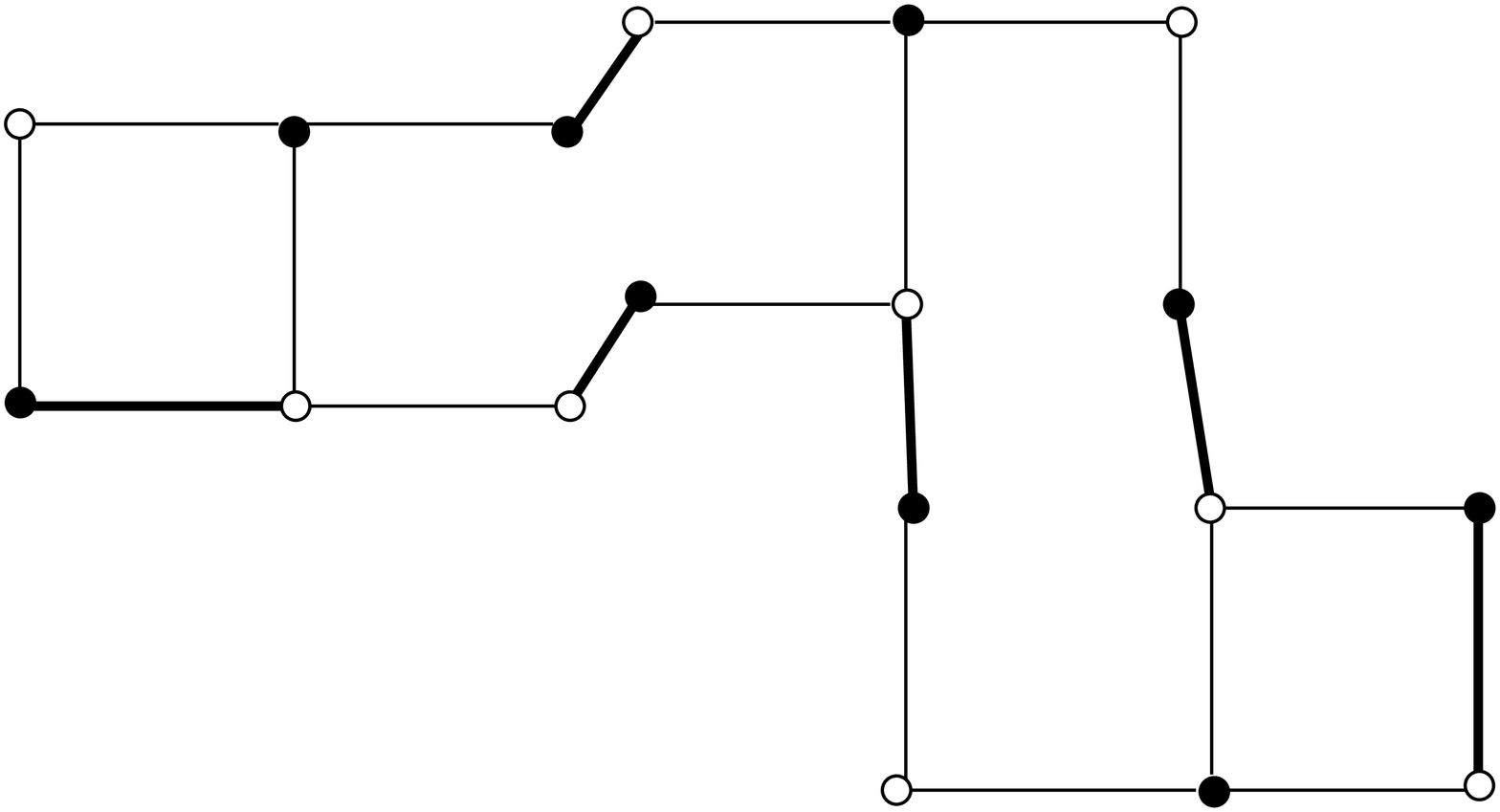}}
\caption{Unraveling the twisted dimer graph}\label{unravel}
\end{figure}

We then find a spanning tree in the graph $\Gamma'''$ indicated below by the solid line segments transverse to the edges. We assign a $\epsilon$- weighting of $+1$ to all of these edges. This is shown in Figure \ref{tree}.
\begin{figure}[ht]
\includegraphics[width=2.4in]{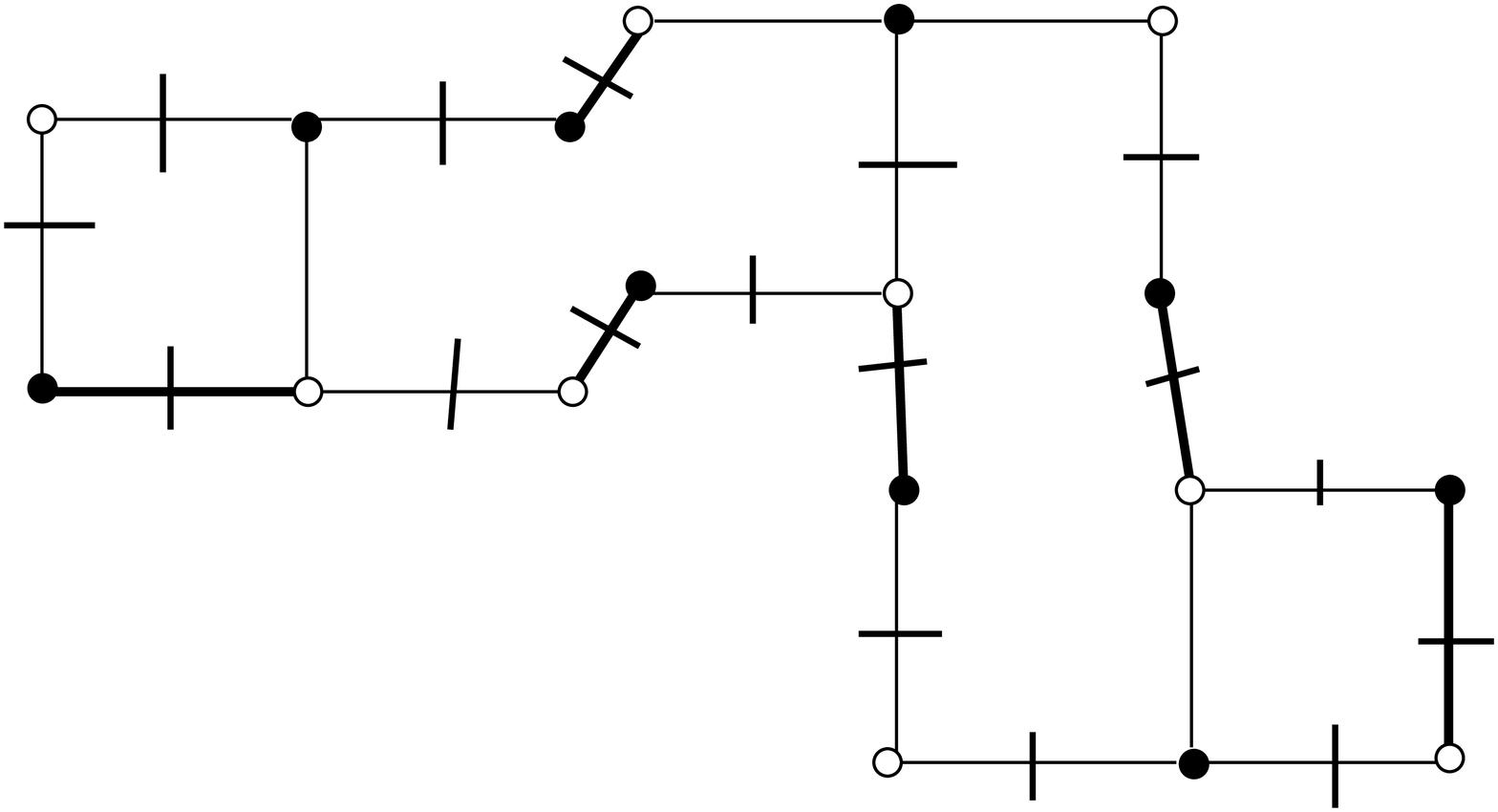}
\caption{Finding a spanning tree}\label{tree}
\end{figure}

Finally, as is shown in Figure \ref{comkas}, we complete $\epsilon$ to a Kasteleyn weighting. In the picture $-1$ weights are indicated by double line segments.
\begin{figure}[ht]
\includegraphics[width=2.4in]{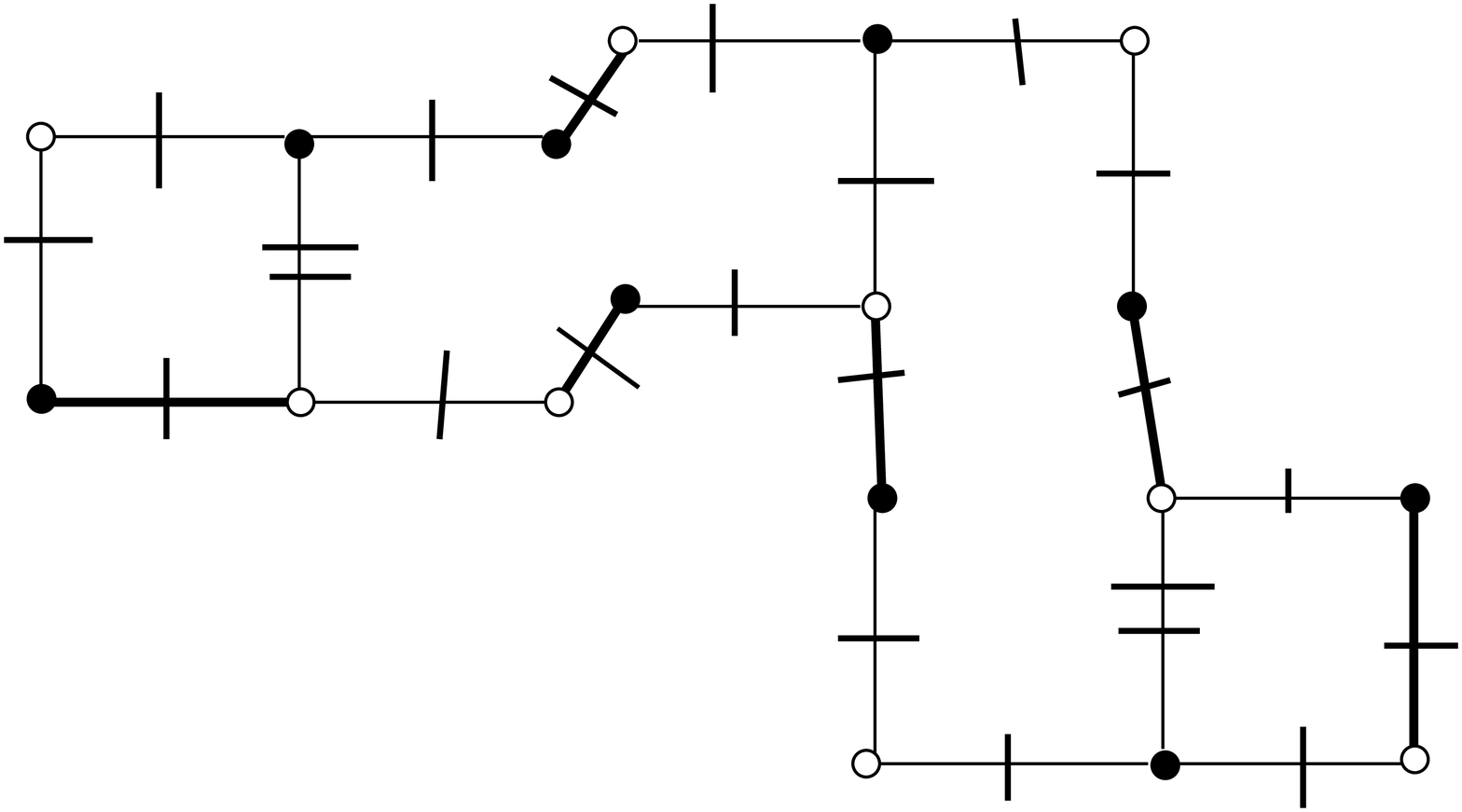}
\caption{A Kasteleyn weighting}\label{comkas}
\end{figure}

Calculating the partition function one can see that the polynomial obtained is $$\Delta_{\textup{Tref}, \rho}(t) = t^6-t^5-t^4+2t^3-t^2-t+1.$$  which agrees up to multiplication by $-1$ with Example \ref{trefex}.

\end{example}

\bibliography{Dimers}
\bibliographystyle {amsalpha}

\end{document}